\documentclass[a4paper, 12pt, oneside, reqno]{amsart}
\calclayout

\usepackage[utf8]{inputenc}
\usepackage[T1]{fontenc}
\usepackage[english]{babel}
\usepackage{amsmath, amsfonts, amssymb, amsthm, mathtools, indentfirst, bbm}
\usepackage[dvipsnames]{xcolor}
\usepackage[shortlabels]{enumitem}
\usepackage[languagenames,fixlanguage]{babelbib}
\usepackage{commath}
\usepackage{textcomp}
\usepackage{mathrsfs}
\usepackage{todonotes}
\allowdisplaybreaks

\usepackage[unicode=true, pdfstartview={XYZ null null 1.00}, colorlinks=true, linkcolor=blue, urlcolor=blue, citecolor=purple]{hyperref}

\newcommand{\bbN}{\mathbb{N}}
\newcommand{\bbZ}{\mathbb{Z}}
\newcommand{\bbR}{\mathbb{R}}

\newcommand{\bbP}{\mathbb{P}}
\newcommand{\bbE}{\mathbb{E}}
\newcommand{\bbjedan}{\mathbbm{1}}

\newcommand{\calE}{\mathcal{E}}
\newcommand{\calF}{\mathcal{F}}

\newcommand{\calD}{\mathcal{D}}
\newcommand{\calP}{\mathcal{P}}

\newcommand{\aps}[1]{\vert #1 \vert}
\newcommand{\APS}[1]{\left\vert #1 \right\vert}
\newcommand{\floor}[1]{\lfloor #1 \rfloor}
\newcommand{\FLOOR}[1]{\left\lfloor #1 \right\rfloor}

\newcommand{\OBL}[1]{\left( #1 \right)}
\newcommand{\UGL}[1]{\left[ #1 \right]}

\let\ge\geqslant
\let\le\leqslant

\newtheorem{DEF}{Definition}[section]
\newtheorem{TM}[DEF]{Theorem}
\newtheorem{PROP}[DEF]{Proposition}
\newtheorem{LM}[DEF]{Lemma}
\newtheorem{COR}[DEF]{Corollary}

\newtheorem{Claim}{Claim}

\numberwithin{equation}{section}

\makeatletter
\def\@tocline#1#2#3#4#5#6#7{\relax
  \ifnum #1>\c@tocdepth 
  \else
    \par \addpenalty\@secpenalty\addvspace{#2}%
    \begingroup \hyphenpenalty\@M
    \@ifempty{#4}{%
      \@tempdima\csname r@tocindent\number#1\endcsname\relax
    }{%
      \@tempdima#4\relax
    }%
    \parindent\z@ \leftskip#3\relax \advance\leftskip\@tempdima\relax
    \rightskip\@pnumwidth plus4em \parfillskip-\@pnumwidth
    #5\leavevmode\hskip-\@tempdima
      \ifcase #1
       \or\or \hskip 1em \or \hskip 2em \else \hskip 3em \fi%
      #6\nobreak\relax
    \dotfill\hbox to\@pnumwidth{\@tocpagenum{#7}}\par
    \nobreak
    \endgroup
  \fi}
\makeatother

\title[Transition probability estimates]{Transition probability estimates\\ for subordinate random walks}
	\author{Wojciech Cygan}
		\address{
		Wojciech Cygan,
		Institut f\"{u}r Mathematische Stochastik,
		Technische Universit\"{a}t \newline Dresden, Germany 
		\& 
		Instytut Matematyczny,
		Uniwersytet Wrocławski, Poland
		}
		\email{wojciech.cygan@uwr.edu.pl}
	\author{Stjepan Šebek}
		\address{
		Stjepan Šebek,
		Department of Applied Mathematics,
		Faculty of Electrical\newline Engineering and Computing,
							University of Zagreb, 
							Croatia}
		\email{stjepan.sebek@fer.hr}
\date{}

\begin{document}

\begin{abstract}
Let $S_n$ be the simple random walk on the integer lattice $\mathbb{Z}^d$. For a Bernstein function $\phi$ we consider a random walk $S^\phi_n$ which is subordinated to $S_n$. Under a certain assumption on the behaviour of $\phi$ at zero we establish global estimates for the transition probabilities of the random walk $S^\phi_n$. The main tools that we apply are the parabolic Harnack inequality and appropriate bounds for the transition kernel of the corresponding continuous time random walk. 
\end{abstract}


\maketitle

\section{Introduction}

The main aim of this article is to obtain global estimates for the transition probabilities for a class of random walks on the integer lattice $\mathbb{Z}^d$ that are subordinated to the simple random walk. Random walks from this class are obtained via discrete subordination which was defined in \cite{BSC}. They have neither second moment nor finite support and thus studying their long time behaviour becomes very demanding. 
The procedure of discrete subordination can be regarded as a discrete counterpart of the Bochner's subordination for semigroups of operators which was widely applied in probability theory for continuous time Markov processes.

To be more precise, let $P$ be the one-step transition operator of the simple (symmetric) random walk $S_n$ on the space $\mathbb{Z}^d$, that is $Pf(x)= \frac{1}{2d}\sum_{j=1}^df(x\pm e_j)$, where $e_j$ is the unit vector in $\mathbb{Z}^d$ with $j^{\mathrm{th}}$ component $1$. For any Bernstein function $\phi$ such that $\phi (0)=0$ and $\phi (1)=1$ we define a new transition operator $P^\phi$ via the following functional equation
\begin{align*}
I - P^\phi = \phi (I-P).
\end{align*}
The operator $P^\phi -I$ generates a random walk $S^\phi _n$ which is the subordinate random walk related to the function $\phi$, see Section \ref{sec:Prem} for the probabilistic definition.

 In this article we are concerned with the transition probabilities of the random walk $S^\phi _n$ which are defined as 
$p^\phi (n,x,y) = \mathbb{P}^x (S^\phi _n =y)$. In the course of study we assume that $\phi$ is a complete Bernstein function 
and satisfies the following \textit{scaling condition}: there are  some constants $c_*, c^* > 0$ and $0 < \alpha_* \le \alpha^* < 1$ such that
\begin{equation}\label{eq:scaling}
	c_* \OBL{\frac{R}{r}}^{\alpha_*} \le \frac{\phi(R)}{\phi(r)} \le c^* \OBL{\frac{R}{r}}^{\alpha^*}, \quad 0 < r \le R \le 1.
\end{equation}
Under these two assumptions we establish global estimates for the function $p^\phi (n,x,y) $, that is we prove that for all $x,y\in \mathbb{Z}^d$ and $n\in \mathbb{N}$ it holds
\begin{align}\label{eq:main_result}
p^\phi (n,x,y)\asymp \min  \Big\{ \big(\phi^{-1}(n^{-1})\big)^{d/2},\,  \frac{n\,   \phi(\aps{x - y}^{-2})}{\aps{x - y}^d}\Big\},
\end{align}
see Theorem \ref{tm:on-diag_bounds}, Theorem \ref{thm:lower_bound} and Theorem \ref{tm:upper_bound_for_Sphi}.
In the above relation, the symbol $\asymp$ means that the ratio of the two expressions is bounded from below and from above by some positive constants.

Similar questions have already been addressed in the literature. In \cite{BL02} global estimates for the transition probabilities of random walks with unbounded range on $\mathbb{Z}^d$ were established under the assumption that one-step transition probability from $x$ to $y$ is a stable-like function, i.e.\ it is comparable to the regularly varying function $|x-y|^{-(d+\alpha)}$
for $\alpha \in (0, 2)$. Let us compare this result to our estimates \eqref{eq:main_result}.  
In \cite{MS18} it was proved  that under the same assumptions as in the present paper, the one-step transition probability of the subordinate random walk $S^\phi_n$ satisfies 
\begin{equation}\label{eq:one_step_prob_estimates}
	p^{\phi}(1,x,y) \asymp \aps{x - y}^{-d} \phi(\aps{x - y}^{-2}), \quad  \mathrm{for}\  x \neq y.
\end{equation}
In particular, if $\alpha_* = \alpha^*$ in \eqref{eq:scaling}
then we are in the scope of  \cite{BL02} but it may well happen that $0 < \alpha_* < \alpha^* < 1$. Moreover, 
condition \eqref{eq:scaling} means that the function $\phi$ is a $O$-regularly varying function at $0$ with Matuszewska indices contained in $(0,1)$, see \cite[Sec. 2]{BGT_book}. Complete Bernstein functions with such behaviour at zero can be found in the closing table of \cite{BFs_book} and include functions: $\phi (\lambda) = \lambda^\alpha +\lambda^\beta$, $\alpha , \beta \in (0,1)$; $\phi (\lambda) = \lambda^\alpha (\log (1+\lambda))^\beta$, $\alpha \in (0,1),\ \beta \in (0,1-\alpha)$; $\phi (\lambda)= (\log (\cosh( \sqrt{\lambda} )))^\alpha$, for $\alpha \in (0,1)$, etc. It is also possible to construct examples of complete Bernstein functions which fulfil \eqref{eq:scaling} and are not comparable to any regularly varying function, see e.g.\  \cite{KImSongVondr_SCi}. This shows that our estimates apply to a class of random walks whose one-step transition probabilities may not be comparable to a regularly varying function which goes beyond the assumptions of \cite{BL02}.

In \cite{MSC15} the authors found global estimates for transition probabilities for a class of Markov chains on a uniformly discrete metric measure space under the assumption that the one-step transition kernel is comparable to a regularly varying function times a homogenuos volume groth function. We mention here further related papers and monographs which focus on estimates of transition probabilities of random walks \cite{Alexop}, \cite{Barlow_book}, \cite{BarlBasKumag}, \cite{HebischSC}, \cite{La96}, \cite{Lawler_Limic_book}, \cite{Spitzer}, \cite{Trojan}, \cite{Woess_book}.

A class of subordinate random walks was introduced in paper \cite{BSC} in the context of random walks on groups. In \cite{BCT17} authors established asymptotics of the transition probabilities of subordinate random walks on $\mathbb{Z}^d$ under the assumption that $\phi$ is regularly varying at zero. This result is valid in a specific region which depends on the time and the space variables. It implies that the corresponding estimates hold only in that region whereas \eqref{eq:main_result} is true for all $x\in \mathbb{Z}^d$ and $n\in \mathbb{N}$. There are more papers where subordinate random walks were studied from potential-theoretic point of view, see
\cite{BC1}, \cite{BC2}, \cite{Mimica}, \cite{MS18} and \cite{Deng}.
Note that discrete subordination allows us to efficiently construct examples of random walks with the controlled tail behaviour. In particular, the regular variation at zero of the function $\phi$ is a necessary and sufficient condition for $S_n^\phi$ to belong to the domain of attraction of a stable law, see \cite{Mimica} and \cite{BCT17}.

Let us comment on the structure and methods of the article. 
In Section \ref{sec:Prem} we give the precise definition of the subordinate random walk and we prove some auxiliary results which include an estimate for the time to leave a ball for the random walk $S^\phi _n$. Our proof is an application of the concentration inequality from \cite{Pr81}. 
Section \ref{sec:On-diag} is devoted to the proof of the on-diagonal bound for the kernel $p^\phi (n,x,y)$. For this we use the  Fourier analytic approach which was previously applied in \cite{BCT17} to find asymptotics of $p^\phi (n,x,y)$ under the assumption that $\phi$ is a regularly varying function at zero. 
In Section \ref{sec:Parab} we prove a parabolic Harnack inequality which is in itself a valuable contribution and this is the  main tool that we use to obtain off-diagonal bounds for $p^\phi (n,x,y)$. To show this inequality we follow the elegant approach of \cite{BL02}, which was also applied in \cite{MSC15}. 
In Section \ref{sec:lower} we obtain the global lower bound by the application of the parabolic Harnack inequality combined with the on-diagonal estimate.
Section \ref{sec:off-diagonal_bounds} is a twofold paragraph. In the first part we study the continuous time random walk which is constructed from $S^\phi _n$ with the aid of an independent Poisson process. For such a process we find the upper heat kernel estimate. To get this result we apply the marvellous approach of \cite{CKW16} where the authors study stability of heat kernel estimates for jump processes on metric measure spaces.
In the second part we apply estimates for the continuous time random walk to prove hitting time estimates and, finally, upper bounds for $p^\phi (n,x,y)$.

\subsection*{Notation}
Throughout the paper $C, c, c_1, c_2, \ldots$ will denote absolute constants. Their labelling starts anew in each statement and their dependence on the function $\phi$ and on the dimension $d$ will not be mentioned explicitly. 
The cardinality of a set $A\subset \mathbb{Z}^d$ is denoted by $\aps{A}$. 
The Euclidean distance between $x$ and $y$ is denoted by $\aps{x - y}$. 
For $x \in \bbR^d$ and $r > 0$, we write $B(x, r) = \{y \in \bbZ^d : \aps{y - x} < r\}$ and $B_r=B(0, r)$.
We use notation $a \wedge b := \min\{a, b\}$ and $a \vee b := \max\{a, b\}$. 
For any two positive functions $f$ and $g$, we write $f \asymp g$ if there exist constants $c_1, c_2 > 0$ such that $c_1 \le g/f \le c_2 $. 
\section{Preliminaries}\label{sec:Prem}
Let $S_n =  X_1 + \cdots + X_n$ be the simple (symmetric) random walk in $\bbZ^d$ which starts from the origin. This means $(X_k)_{k \ge 1}$ is a sequence of independent identically distributed random variables defined on a given probability space $(\Omega, \mathscr{A}, \bbP)$ with distribution $\bbP(X_k = e_i) = \bbP(X_k = -e_i) = 1/2d$, for each $i = 1, 2, \ldots, d$. Here $e_i$ is the $i^{\mathrm{th}}$ unit vector in $\bbZ^d$.  

Let $\phi$ be a Bernstein function such that $\phi (0)=0$, $ \phi(1)=1$. Such a function admits the following integral representation
\begin{equation}\label{eq:def_of_phi}
	\phi(\lambda) = \ell \lambda + \int_{( 0, \infty )} \big( 1 - e^{-\lambda t} \big) \mu (\dif t),
\end{equation}
for $\ell \ge 0$ and a measure $\mu$ on $( 0,\infty )$ satisfying $\int_{( 0,\infty )} {(1 \wedge t) \mu(\dif t)} < \infty$, see \cite[Sec. 3]{BFs_book}. 

We consider a sequence of positive numbers $a_m^\phi$ which is related to the function $\phi$ and is defined as
\begin{align}\label{eq:def_of_c_m^phi}
	a_m^{\phi} & = \ell \delta_1(m)+ \frac{1}{m!}\int_{( 0, \infty )}{t^m e^{-t}\mu (dt)},\quad m \geq  1,
\end{align}
where $\delta_x$ is the Diraac measure at $x$.
One easily verifies that
\begin{equation*}
	\sum_{m = 1}^{\infty} {a_m^{\phi}} = \ell + \int_{( 0, \infty )} {(e^t - 1) e^{-t} \mu(\dif t)} = \ell + \int_{( 0, \infty )} {(1 - e^{-t}) \mu(\dif t)} = \phi(1) = 1.
\end{equation*}
Let $\tau _n = R_1+\cdots +R_n$ be a random walk on $\mathbb{Z}_+$ with increments $R_i$ that are independent of the random walk $S_n$ and have the distribution given by $\bbP(R_1 = m) = a_m^{\phi}$.
A subordinate random walk is defined as $S^{\phi}_n := S_{\tau  _n}$, for all  $n \ge 0$. Such random walks were introduced in \cite{BSC} and later studied in papers \cite{BC1}, \cite{BC2}, \cite{Mimica}, \cite{BCT17}, \cite{MS18}, see also \cite{Deng}.
Notice that the one-step transition probability $p^{\phi}(1, x,y) $ of the random walk $S^{\phi}_n$ is of the form
\begin{align}\label{eq:distribution_of_X1}
	p^{\phi}(1, x, y) = \bbP^x(S^{\phi}_1 = y)
	 = \sum_{m = 1}^{\infty} {\bbP^x(S_{R_1} = y \mid R_1 = m)} a_m^{\phi} 
	 =  \sum\limits_{m = 1}^{\infty} p(m,x,y) a_m^{\phi},
\end{align}
where $p(n, x,y) = \bbP^x(S_n = y)$ stands for the $n$-step transition probability 
of the simple random walk $S_n$. We use the notation $p^{\phi}(n,x,x)=p^{\phi}(n,0)$ and $p^{\phi}(1,x,y)=p^{\phi}(x,y)=p^{\phi}(x-y)$.

In the course of study we always assume that $\phi$ is a complete Bernstein function. Recall that this means that the measure $\mu$ from \eqref{eq:def_of_phi} has a completely monotone density with respect to Lebesgue measure, see \cite[Def. 6.1.]{BFs_book}. We additionally require that $\phi$ has no drift term, that is $\ell = 0$ in \eqref{eq:def_of_phi}. Next assumption on the function $\phi$ is that it satisfies \textit{scaling condition} \eqref{eq:scaling}.
These assumptions will not be explicitly stated in the results.
%

\subsection{Auxiliary results}
We will repeatedly use the fact that 
\begin{equation}\label{eq:volume_growth}
	c' r^d \le \aps{B(x, r)} \le c'' r^d,\quad x\in \mathbb{Z}^d,
\end{equation}
for constants $c',c''>0$ which depend only on the dimension $d$.

We recall that for any Bernstein function $\phi$ it holds $\phi(\lambda t) \le \lambda \phi(t)$, for all $\lambda \ge 1$, $t > 0$, which implies
\begin{equation}\label{eq:phi(v)/phi(u)_le_v/u}
	\frac{\phi(v)}{\phi(u)} \le \frac{v}{u}, \quad 0 < u \le v.
\end{equation}

We formulate bounds for the inverse function $\phi^{-1}$ which easily follow from \eqref{eq:scaling} and take the form
\begin{equation}\label{eq:scaling_for_inverse}
	(1/c^*)^{1/\alpha^*} \left(\frac{R}{r}\right)^{1/\alpha^*} \le \frac{\phi^{-1}(R)}{\phi^{-1}(r)} \le (1/c_*)^{1/\alpha_*} \left(\frac{R}{r}\right)^{1/\alpha_*}, \quad 0 < r \le R \le 1.
\end{equation}

Throughout the paper we use the following decreasing function
\begin{align}\label{j_function}
j(r) = r^{-d} \phi(r^{-2}), \quad r> 0.
\end{align}
Notice that with this notation \eqref{eq:one_step_prob_estimates} becomes $p^{\phi}(1, x, y) \asymp j(\aps{x - y})$, $x \neq y$.

\begin{LM}\label{lm:Lemma21_from_CKW}
	There exists a constant $c_0 > 0$ such that
	\begin{equation*}
		\sum_{y \in B(x, r)^c} j(\aps{x - y}) \le c_0 \phi(r^{-2})
	\end{equation*}
	for every $x \in \bbZ^d$ and $r > 0$.
\end{LM}
\begin{proof}
Assume that $r \ge 1$. By \eqref{eq:scaling}, we have
	\begin{align*}
		\sum_{y \in B(x, r)^c} j(\aps{x - y})
		& \le 
		\sum_{i = 0}^{\infty} \sum_{2^{i} r \le |x-y| < 2^{i+1}r }  j(2^i r) \\
		& \le  c'' 2^d \phi(r^{-2}) \sum_{i = 0}^{\infty} \frac{\phi((2^i r)^{-2})}{\phi(r^{-2})} 
		 \le c_0 \phi(r^{-2}).
	\end{align*}
If $r \in (0, 1)$ then $B(x, r)^c = B(x, 1)^c$. Therefore
	\begin{equation*}
		\sum_{y \in B(x, r)^c} j(\aps{x - y}) = \sum_{y \in B(x, 1)^c} j(\aps{x - y}) \le c_0 \phi(1^{-2}) \le c_0 \phi(r^{-2}),
	\end{equation*}
what finishes the proof.
\end{proof}

Next we prove a pair of useful estimates for the subordinate random walk.
\begin{LM}\label{lm:lower_bound_for_P(X_1=0)}
	There exists a constant $C_1 > 0$ such that 
\begin{align*}
p^{\phi}(x, x) \ge C_1,\quad x \in \bbZ^d.
\end{align*}	
\end{LM}
\begin{proof}
	By \cite[Thm. 1.2.1]{La96}, 
	\begin{equation*}
		\bbP (S_{2m} = 0) \asymp m^{-d/2}, m\in \bbN.
	\end{equation*}
This and the fact that $\bbP (S_{2m - 1} = 0) = 0$ combined with \eqref{eq:distribution_of_X1}, \cite[Lemma 3.1.]{MS18} and \eqref{eq:scaling} yield for all $x \in \bbZ^d$
	\begin{align*}
		p^{\phi}(x, x)
		 \ge c_1 \sum_{m = 1}^{\infty} \frac{\phi((2m)^{-1})}{2m} m^{-d/2}  \nonumber 
		 \ge 
		\frac{c_1}{c^*2^{\alpha^*+1}} \sum_{m = 1}^{\infty}  m^{-\alpha^* - d/2 - 1} >0,
	\end{align*}
as desired.
\end{proof}

\subsection*{Estimates for probability of leaving a ball}
In this paragraph we establish the following result.
\begin{TM}\label{tm:maximal_inequality}
	There exists a constant $\gamma \in (0, 1)$ such that for all $r > 0$
	\begin{equation}\label{eq:maximal_inequality}
		\bbP^x \big(\max_{k \le \floor{\gamma / \phi(r^{-2})}} \aps{S^{\phi}_k - x} \ge r/2\big) \le 1/4.
	\end{equation}
\end{TM}
Our approach is based on the application of the concentration inequality from \cite[Lemma on page 949]{Pr81}. Since we use it at numerous occasions in the paper, we formulate it as a separate result.
\begin{LM}\label{lm:Pruit_bound_general_form}
	Let $(X_k)_{k \ge 1}$ be a sequence of independent random variables taking values in $\bbR^d$ and with the common distribution function denoted by $F$. Let $S_n = \sum_{i = 1}^n X_i$ be the corresponding random walk and set $M_n = \max_{k \le n} \aps{S_k}$. Define for $x > 0$
	\begin{equation*}
		G(x) = \bbP(\aps{X_1} > x), \quad K(x) = x^{-2} \int_{\aps{y} \le x} \aps{y}^2 dF(y),
	\end{equation*}
	\begin{equation*}
		M(x) = x^{-1} \APS{\int_{\aps{y} \le x} ydF(y)}, \quad	h(x) = G(x) + K(x) + M(x).
	\end{equation*}
	Then there is a constant $C > 0$, which depends only on the dimension $d$, such that for all $n \in \bbN$ and all $a > 0$
	\begin{equation*}
		\bbP(M_n \ge a) \le Cn h(a) \quad \mathrm{and}\quad \bbP(M_n \le a) \le \frac{C}{n h(a)}.
	\end{equation*}
\end{LM}
Since the random walk $S^\phi_n$ is symmetric, the associated function $h$ is of the form
\begin{equation}\label{eq:def_of_h}
	h(x) = \bbP(\aps{S^{\phi}_1} > x)+ x^{-2} \int_{\aps{y} \le x} \aps{y}^2 dF(y),
\end{equation}
where $F$ is the distribution function of the random variable $S^{\phi}_1$. Before we prove Theorem \ref{tm:maximal_inequality} we show that under the scaling condition \eqref{eq:scaling} the function $h$ is dominated by the function $\phi $.

\begin{LM}\label{lm:h_le_phi}
In the above notation, there is a constant $C \ge 1$ such that
	\begin{equation*}
		h(x) \le C \phi(x^{-2}), \quad x > 0.
	\end{equation*}
\end{LM}
\begin{proof}
First observe that if $x\in (0,1)$ then $h(x) = \bbP(S^{\phi}_1 \neq 0)$ and whence the result follows. Assume next that $x\geq 1$. 
Using \eqref{eq:one_step_prob_estimates} and \eqref{eq:scaling} we get
	\begin{align*}
	\bbP(\aps{S^{\phi}_1} > x) &\le c_1 \sum_{\aps{y} > x} \aps{y}^{-d} \phi(\aps{y}^{-2}) 
		\le \frac{c_1}{c_*} \phi(x^{-2}) \sum_{\aps{y} > x} \aps{y}^{-d} \left(x/\aps{y}\right)^{2 \alpha_*} \\
		&\le c_2 x^{2 \alpha_*} \phi(x^{-2}) \int_x^{\infty} r^{-d - 2\alpha_*} r^{d - 1} \, dr 
		=c_3 \phi(x^{-2}).
	\end{align*}
We can similarly show that 
\begin{align*}
x^{-2} \int_{\aps{y} \le x} \aps{y}^2 dF(y)\le c_4 \phi(x^{-2})
\end{align*}
for some constant $c_4 > 0$ and the proof is finished.
\end{proof}

\begin{proof}[Proof of Theorem \ref{tm:maximal_inequality}]
	We first consider the case $r < 1$. Since $\phi$ is increasing and $\phi(1) = 1$, we have $\gamma / \phi(r^{-2}) < 1$, for any $\gamma \in (0, 1)$. Therefore
	\begin{equation*}
		\max_{k \le \floor{\gamma / \phi(r^{-2})}} \aps{S^{\phi}_k - x} = \aps{S^{\phi}_0 - x}
	\end{equation*}
	and thus for any $r < 1$ it holds
	\begin{equation*}
		\bbP^x \OBL{\max_{k \le \floor{\gamma / \phi(r^{-2})}} \aps{S^{\phi}_k - x} \ge r/2} 
		= 0.
	\end{equation*}
Assume that $r \ge 1$. Applying Lemma \ref{lm:Pruit_bound_general_form} we get
	\begin{equation}\label{Pruitt_bound}
		\bbP^x \OBL{\max_{k \le \floor{\gamma / \phi(r^{-2})}} \aps{S^{\phi}_k - x} \ge r/2} \le c_1 \floor{\gamma / \phi(r^{-2})} h(r/2),
	\end{equation}
	where $c_1$ depends only on the dimension $d$. By Lemma \ref{lm:h_le_phi} and \eqref{eq:phi(v)/phi(u)_le_v/u}, 
	\begin{equation*}
	 	\bbP^x \OBL{\max_{k \le \floor{\gamma / \phi(r^{-2})}} \aps{S^{\phi}_k - x} \ge r/2} \le 4 c_1 C \floor{\gamma / \phi(r^{-2})} \phi(r^{-2}) \le 4c_1 C \gamma.
	 \end{equation*}
	Choosing $\gamma = \frac{1}{2}\wedge \frac{1}{16c_1C}$ we obtain \eqref{eq:maximal_inequality} for all $r > 0$.
\end{proof}

\section{On-diagonal bounds}\label{sec:On-diag}
In this section we establish the on-diagonal bounds. We apply a Fourier analytic method which is extracted from \cite{BCT17}.

\begin{TM}\label{tm:on-diag_bounds}
For all $n \in \bbN$ it holds
	\begin{equation}\label{eq:on-diag_bounds}
		p^{\phi}(n, 0) \asymp \OBL{\phi^{-1}(n^{-1})}^{d/2}.
	\end{equation}
	\end{TM}
\begin{proof}
Let $\Psi$ be the characteristic function of the simple random walk $S$. Then the characteristic function of $S^\phi$ is $\Psi^{\phi}(\theta) = 1 - \phi(1 - \Psi(\theta))$, see \cite{BSC}. Thus, by the Fourier inversion formula,
	\begin{equation}\label{eq:P(Xn=0)-Fourier}
		p^{\phi}(n, 0)  = \frac{1}{(2\pi)^d} \int_{\calD_d} (1 - \phi(1 - \Psi(\theta)))^n d\theta,
	\end{equation}
	where $\calD_d = [-\pi, \pi)^d$.
We fix $\varepsilon >0$	 
and first we estimate the integral in \eqref{eq:P(Xn=0)-Fourier} over the set $\calD_d^{\varepsilon} := \{\theta \in \calD_d : \aps{\theta} \ge \varepsilon\}$. Since $\aps{1 - \phi(1 - \Psi(\theta))} = 1 $ if and only if $\theta \in 2\pi \bbZ^d$, see \cite[Claim 2]{BCT17}, it holds that $\aps{1 - \phi(1 - \Psi(\theta))} < 1 - \eta$ for all $\theta \in \calD_d^{\varepsilon}$ and for some $\eta \in (0, 1)$. 
Hence   
	\begin{equation*}
		\frac{1}{(2\pi)^d} \int_{\calD_d^{\varepsilon}} \aps{1 - \phi(1 - \Psi(\theta))}^n d\theta \le (1 - \eta)^n.
	\end{equation*}

Next, we consider the remaining part of the integral in \eqref{eq:P(Xn=0)-Fourier}, that is over the ball $B_\varepsilon $. We set $a_n = \OBL{\phi^{-1}(n^{-1})}^{1/2}$ and by the change of variable we get
	\begin{equation*}
		a_n^{-d} \int_{\aps{\theta} < \varepsilon} \OBL{1 - \phi(1 - \Psi(\theta))}^n d\theta = \int_{\aps{\xi} < \varepsilon / a_n} \OBL{1 - \phi(1 - \Psi(a_n \xi))}^n d\xi.
	\end{equation*}
To finish the proof we need to show that for some $c_1,c_2>0$
\begin{align}\label{eq:change_of_variable}
c_1\le \int_{\aps{\xi} < \varepsilon / a_n} \OBL{1 - \phi(1 - \Psi(a_n \xi))}^n d\xi \le c_2.
\end{align}
 Notice that it suffices to prove \eqref{eq:change_of_variable} only for $n$ large enough, as the integrand in  \eqref{eq:change_of_variable}  is strictly positive if $\varepsilon$ is small enough, and thus in the end of the proof we can change constants appropriately to estimate the  expression in \eqref{eq:P(Xn=0)-Fourier} for all $n$.
 
We observe that
\begin{align}\label{al:lim=1/2}
\lim_{n \to \infty} \frac{1 - \Psi(a_n \xi)}{\aps{a_n \xi}^2 / d} = \frac{1}{2}.
\end{align}
Indeed, this follows easily from the fact that 
	\begin{equation*}
		\Psi(\theta) = \frac{1}{d} \sum_{m = 1}^d \cos(\theta_m), \quad \theta = (\theta_1, \theta_2, \ldots, \theta_d)
	\end{equation*}
	and, for some $c_3 > 0$ and for all $x \in \bbR$,
		\begin{equation*}
		\aps{1 - \cos(x) - x^2 / 2} \le c_3 x^4.
	\end{equation*}
	
We next prove that for some $c_4, c_5>0$ and for all $n\in \mathbb{N}$
	\begin{equation}\label{eq:n_phi_final}
		c_4 \left( \aps{\xi}^{2\alpha_*} \wedge \aps{\xi}^{2\alpha^*}\right) \le n \phi(1 - \Psi(a_n \xi)) 
		\le c_5 \left( \aps{\xi}^{2\alpha_*}\vee \aps{\xi}^{2\alpha^*}\right).
	\end{equation}
For that we establish the following simple result.
	\begin{Claim}\label{cl:phi_over_phi}
		Let $(a_n)$ and $(b_n)$ be two sequences of positive numbers both tending to zero and such that $\lim_{n \to \infty} (a_n / b_n) = 1$. Then there exists a constant $c_6 > 0$ such that
		\begin{equation}\label{eq:phi_over_phi}
			c_6^{-1} \le \frac{\phi(a_n)}{\phi(b_n)} \le c_6, \quad  n\in \bbN.
		\end{equation}
	\end{Claim}
\noindent
	\textit{Proof of Claim \ref{cl:phi_over_phi}}. Scaling condition \eqref{eq:scaling} implies that, for some $c_7>0$,
	\begin{equation*}
		c_7^{-1} \left( (x/y)^{\alpha_*} \wedge (x/y)^{\alpha^*}\right)
		\le 
		\frac{\phi (x)}{\phi (y)}
		 \le c_7 \left( (x/y)^{\alpha_*} \vee (x/y)^{\alpha^*}\right), \quad  x, y \in (0, 1).
	\end{equation*}
	With this inequality it is straightforward to obtain \eqref{eq:phi_over_phi}.\\
\vspace*{0,1cm}

\noindent By Claim \ref{cl:phi_over_phi} and \eqref{al:lim=1/2},
	\begin{equation*}
		c_8^{-1} \le \frac{\phi\OBL{1 - \Psi(a_n \xi)}}{\phi\OBL{\aps{a_n \xi}^2 / 2d}} \le c_8
	\end{equation*}
and whence
	\begin{equation}\label{eq:n_phi_part-1}
		n \phi(1 - \Psi(a_n \xi)) = \frac{\phi\OBL{1 - \Psi(a_n \xi)}}{\phi\OBL{\aps{a_n \xi}^2 / 2d}} \frac{\phi\OBL{\aps{a_n \xi}^2 / 2d}}{n^{-1}} \asymp \frac{\phi\OBL{a_n^2 \aps{\xi}^2 / 2d}}{\phi(a_n^2)}.
	\end{equation}
Applying scaling condition \eqref{eq:scaling} in \eqref{eq:n_phi_part-1} we get \eqref{eq:n_phi_final} .

Next, we notice that
	\begin{equation*}
 \lim_{n \to \infty} \frac{n \log\OBL{1 - \phi(1 - \Psi(a_n \xi)}}{-n \phi\OBL{1 - \Psi(a_n\xi)}} = 1.
	\end{equation*}
Thus, by \eqref{eq:n_phi_final}, for $n$ large enough,
	\begin{align*}
\int_{\aps{\xi} < \varepsilon/a_n} \!\!\!\! \!\!\!\! e^{-c_9 \left( \aps{\xi}^{2\alpha_*}\vee \aps{\xi}^{2\alpha^*}\right)}d\xi 
 \le 
 \int_{\aps{\xi} < \varepsilon/a_n} \!\!\!\!  \!\! \OBL{1 - \phi(1 - \Psi(a_n \xi))}^n d\xi  
 \le 
		\int_{\aps{\xi} < \varepsilon/a_n} \!\!\!\!  \!\!\!\! e^{-c_{10} \left( \aps{\xi}^{2\alpha_*}\wedge \aps{\xi}^{2\alpha^*}\right)} d\xi.
\end{align*}
Since both of the side integrals converge to positive constants as $n$ goes to infinity, we conclude that \eqref{eq:change_of_variable} is valid for $n$ large enough and the proof is finished.
\end{proof}
\begin{COR}\label{rem:off-diag_bd_from_on-diag}
There is a constant $c > 0$ such that
	\begin{equation*}
		p^{\phi}(n, x, y) \le c \OBL{\phi^{-1}(n^{-1})}^{d/2},\quad \mathrm{for}\ n \in \bbN\ \mathrm{and}\ x, y \in \bbZ^d.
	\end{equation*}
\end{COR}
\begin{proof}
This follows by Theorem \ref{tm:on-diag_bounds} combined with the Cauchy-Schwarz inequality. 
\end{proof}

\section{Parabolic Harnack inequality}\label{sec:Parab}
In this section we prove the parabolic Harnack inequality which is  the main tool that we will use to obtain off-diagonal bounds in Sections \ref{sec:lower} and \ref{sec:off-diagonal_bounds}. We follow closely the elegant approach of \cite{BL02} but we emphasize that for the case that we undertake in the paper it requires numerous adjustments and alterations.

 Let $\calP = \bbN_0 \times \bbZ^d$ and consider the $\calP$-valued Markov chain $(V_k, S^{\phi}_k)_{k \ge 0}$, where $V_k = V_0 + k$. We write $\bbP^{(j, x)}$ for the law of $(V_k, S^{\phi}_k)$ when it starts from $(j, x)$ and we set $\calF_j = \sigma\{(V_k, S^{\phi}_k) : k \le j\}$. 
A bounded function $q$ defined on $\calP$ is called \textit{parabolic} on a subset $D \subseteq \calP$ if $q(V_{k \wedge \tau_D}, S^{\phi}_{k \wedge \tau_D})$ is a martingale, where $\tau_D$ denotes the exit time of the Markov chain $(V_k, S^{\phi}_k)$ from the set $D$. 
We now prove the following important observation.
\begin{LM}\label{lm:hk_is_parabolic_function}
	For each $n_0 \in \bbN$ and $x_0 \in \bbZ^d$ the function $q(k, x) = p^{\phi}(n_0 - k, x, x_0)$ is parabolic on the set $\{0, 1, 2, \ldots, n_0\} \times \bbZ^d$.
\end{LM}
\begin{proof}
By the Markov property,
\begin{align*}
	\bbE[q(V_{k + 1}, S^{\phi}_{k + 1}) \mid \calF_k]
	& = \bbE^{(V_k, S^{\phi}_k)} [p^{\phi}(n_0 - V_1, S^{\phi}_1, x_0)] \\
	& = \sum_{x \in \bbZ^d} p^{\phi}(1, S^{\phi}_k, x) p^{\phi}(n_0 - V_k - 1, x, x_0) 
	 = q(V_k, S^{\phi}_k),
\end{align*}
where the last equality follows by the semigroup relation.
\end{proof}
We introduce the notation
\begin{equation*}
	Q(k, x, r) = \{k, k +1, \ldots, k + \floor{\gamma / \phi(r^{-2})}\} \times B(x, r),
\end{equation*}
where $\gamma$ is the constant from Theorem \ref{tm:maximal_inequality}. We fix the following two constants
\begin{equation}\label{eq:def_of_B_and_b}
	B =  3\vee (2/c_*)^{1/2\alpha_*}, \quad \quad \quad b = 3\vee \big(\floor{(3/c_*)^{1/\alpha_*}}+1\big).
\end{equation}
The main result of this section is the following theorem.
\begin{TM}\label{tm:PHI}
	There exists a constant $C_{PH} > 0$ such that for every non-negative, bounded function $q$ on $\calP$ which is parabolic on the set $ \{0, 1, 2, \ldots, \floor{\gamma / \phi((\sqrt{b} R)^{-2})}\} \times \bbZ^d$, the following inequality holds
	\begin{equation}\label{eq:PHI}
		\max_{(k, y) \in Q(\floor{\gamma / \phi(R^{-2})}, z, R/B)} q(k, y) \le C_{PH} \min_{w \in B(z, R/B)} q(0, w)
	\end{equation}
	for all $z \in \bbZ^d$ and for $R$ large enough.
\end{TM}

Before we prove this theorem we need to establish a series of lemmas.
Let 
\begin{align*}
\tau(k, x, r) := \min\{l \ge 0 : (V_l, S^{\phi}_l) \notin Q(k, x, r)\}
\end{align*}
and put $\tau(x, r) = \tau(0, x, r)$. We observe that $\tau(k, x, r) \le \floor{\gamma / \phi(r^{-2})} + 1$. 
For a non-empty set $ A \subseteq Q(0, x, r)$, we define
\begin{align*}
A(k) = \{y \in \bbZ^d : (k, y) \in A\}\subset \bbZ^d .
\end{align*}
We now fix a non-empty $ A \subseteq Q(0, x, r)$ such that $A(0) = \emptyset$ and we set
\begin{align*}
N(k, x) = \bbP^{(k, x)} (S^{\phi}_1 \in A(k + 1)) \bbjedan_{A^c} (k, x).
\end{align*}
For any $A\subset \calP$ we also define
\begin{align*}
T_A = \min \{n \ge 0 : (V_n, S^{\phi}_n) \in A\},\  \mathrm{and}\  T_{\emptyset} = \infty .
\end{align*}
\begin{LM}\label{lm:J-martingale}
	In the above notation, let
	\begin{equation*}
		J_n = \bbjedan_A(V_n, S^{\phi}_n) - \bbjedan_A(V_0, S^{\phi}_0) - \sum_{k = 0}^{n - 1} N(V_k, S^{\phi}_k).
	\end{equation*}
	The process $J_{n \wedge T_A}$ is a $\calF$-martingale.
\end{LM}
\begin{proof}
We have
	\begin{align*}
		\bbE[J_{(k + 1) \wedge T_A} &- J_{k \wedge T_A} \mid \calF_k] \\
		&= \bbE[\bbjedan_A(V_{(k + 1)\wedge T_A}, S^{\phi}_{(k + 1)\wedge T_A}) - \bbjedan_A(V_{k\wedge T_A}, S^{\phi}_{k\wedge T_A}) - N(V_{k\wedge T_A}, S^{\phi}_{k\wedge T_A}) \mid \calF_k].
	\end{align*}
	If $T_A\leq k$ then the right-hand side of the identity above is zero. If $T_A  > k$ then
	\begin{align*}
		\bbE[J_{(k + 1) \wedge T_A} - J_{k \wedge T_A} \mid \calF_k]
		& = \bbE[\bbjedan_A(V_{k + 1}, S^{\phi}_{k + 1}) \mid \calF_k] - N(V_k, S^{\phi}_k) \\
		& = \bbP^{(V_k, S^{\phi}_k)} (S^{\phi}_1 \in A(V_k + 1)) - N(V_k, S^{\phi}_k) = 0,
	\end{align*}
as desired.
\end{proof}

\begin{PROP}\label{prop:KS_estimate}
	There exists a constant $\theta_1 \in (0, 1)$ such that
	\begin{equation}\label{eq:KS_estimate}
		\bbP^{(0, x)} (T_A < \tau(x, r)) \ge \theta_1 \aps{A} j(r).
	\end{equation}
\end{PROP}
\begin{proof}
We claim that $\floor{\gamma / \phi(r^{-2})} + 1 \le 2\gamma / \phi(r^{-2})$. Indeed,
we have $A(0) = \emptyset$ and $A \neq \emptyset$ so it follows that $A(k) \neq \emptyset$, for some $k \ge 1$. Thus $\gamma / \phi(r^{-2}) \ge 1$, which clearly  yields the claim.

We first assume that $\bbP^{(0, x)}(T_A \le \tau(x, r)) \ge 1/4$. By \eqref{eq:volume_growth} we get
	\begin{align*}
		\aps{A} j(r)
		& \le c''(\floor{\gamma / \phi(r^{-2})} + 1) \phi(r^{-2}) \le 2c'' \gamma.
	\end{align*}
	Hence
	\begin{equation*}
		\bbP^{(0, x)} (T_A \le \tau(x, r)) \ge \frac{1}{4} = \frac{1}{8c'' \gamma} 2c'' \gamma \ge \frac{1}{8c'' \gamma} \aps{A} j(r).
	\end{equation*}
	
Assume that $\bbP^{(0, x)}(T_A \le \tau(x, r)) < 1/4$. Let $M := T_A \wedge \tau(x, r)$. By Lemma \ref{lm:J-martingale} and the Optional Stopping Theorem, $\bbE[J_M] = \bbE[J_0] = 0$. 
This and the fact that $(0, X_0) \notin A$ imply
	\begin{equation*}
		\bbE^{(0, x)} [\bbjedan_A(M, S^{\phi}_M)] = \bbE^{(0, x)} \Big[\sum_{k = 0}^{M - 1} N(k, S^{\phi}_k)\Big].
	\end{equation*}
	By \eqref{eq:one_step_prob_estimates}, Lemma \ref{lm:lower_bound_for_P(X_1=0)} and using monotonicity of the function $j$, we get that for $(k, w) \in Q(0, x, r)\cap A^c$ 
	\begin{align*}
		N(k, w)
		& = \sum_{y \in A(k + 1) \setminus \{w\}} p^{\phi}(w, y) + p^{\phi}(w, w) \bbjedan_{A(k + 1)}(w) \\
	 & \geq  c_1 j(2r) \aps{A(k + 1) \setminus \{w\}} + C_1\, \bbjedan_{A(k + 1)}(w)
	  \geq c_2  j(r) \aps{A(k + 1)}.
	\end{align*}
Observe that if  $M \ge \floor{\gamma / \phi(r^{-2})}$ then $\sum_{k = 0}^{M - 1} \aps{A(k + 1)} = \aps{A}$. Hence, on the set $\{M \ge \floor{\gamma / \phi(r^{-2})}\}$ we have
	\begin{equation*}
		\sum_{k = 0}^{M - 1} N(k, S^{\phi}_k) \ge \sum_{k = 0}^{M - 1} c_2 \aps{A(k + 1)} j(r)  = c_2 \aps{A} j(r).
	\end{equation*}
Since $\bbP^{(0, x)} (T_A \le \tau(x, r)) = \bbE^{(0, x)} [\bbjedan_A(M, S^{\phi}_M)] $, we get
	\begin{align*}
		\bbP^{(0, x)}(T_A \le \tau(x, r)) 
		& \geq c_2 \aps{A} j(r) \bbP^{(0, x)} (M \ge \floor{\gamma / \phi(r^{-2})}) \\
		 & = c_2 \aps{A} j(r) \left(1 - \bbP^{(0, x)} (T_A < \tau(x, r), T_A < \floor{\gamma / \phi(r^{-2})})\right. \\
		 & \quad \qquad \qquad \qquad - \left.\bbP^{(0, x)} (\tau(x, r) < T_A, \tau(x, r) < \floor{\gamma / \phi(r^{-2})})\right) \nonumber \\
		 & \ge c_2 \aps{A} j(r) \left(1 - \bbP^{(0, x)} (T_A \le \tau(x, r)) \right. \\
		 & \quad \qquad \qquad \qquad - \left.\
		 \bbP^{(0, x)} \left( \tau(x, r) \le \floor{\gamma / \phi(r^{-2})}\right) \right).
	\end{align*}
We notice that
$ \max_{k \le \floor{\gamma / \phi(r^{-2})}} \aps{S^{\phi}_k - x} \ge r/2$ if $\tau(x, r) \le \floor{\gamma / \phi(r^{-2})}$.
Thus \eqref{eq:maximal_inequality} implies
	\begin{equation*}
		\bbP^{(0, x)} \OBL{\tau(x, r) \le \floor{\gamma / \phi(r^{-2})}} \le \bbP^{(0, x)} \big(\max_{k \le \floor{\gamma / \phi(r^{-2})}} \aps{S^{\phi}_k - x} \ge r/2\big) \le 1/4.
	\end{equation*}
We conclude the desired result with $\theta_1 = \frac{1}{2}\wedge \frac{1}{8c''\gamma}\wedge \frac{c_2}{2}$.
\end{proof}

\begin{LM}\label{lm:theta_2}
	There exists a constant $\theta_2 > 0$ such that for $(k, x) \in Q(0, z, R/2)$ and for $r > 0$ such that $k \ge \floor{\gamma / \phi(r^{-2})} + 1$ we have
	\begin{equation*}
		\bbP^{(0, x)} \OBL{T_{U(k, x, r)} < \tau(z, R)} \ge \theta_2 \frac{j(R)}{j(r)},
	\end{equation*}
	where $U(k, x, r) = \{k\} \times B(x, r)$.
\end{LM}


\begin{proof}
	Let $Q' = \{k, k - 1, \ldots, k - \floor{\gamma / \phi(r^{-2})}\} \times B(x, r/2)$. One easily verifies that $Q'(0) = \emptyset$ and $Q' \subseteq Q(0, z, R)$. By Proposition \ref{prop:KS_estimate}, we get
	\begin{align*}
		\bbP^{(0, x)} \left(T_{Q'}  < \tau(z, R)\right) 
		& \ge \theta_1 \aps{Q'} j(R) 
		 \ge \theta_1 c' (\floor{\gamma / \phi(r^{-2})} + 1) (r/2)^d j(R) \nonumber \\
		 &\ge \frac{\theta_1 c'}{2^d} \frac{\gamma}{\phi(r^{-2})} r^d j(R) = c_1 \frac{j(R)}{j(r)}.
	\end{align*}
	The strong Markov property yields
	\begin{align}\label{al:strong_Markov_property}
		\bbP^{(0, x)}
		 \OBL{T_{U(k, x, r)} < \tau(z, R)} &\ge \bbP^{(0, x)} \OBL{T_{U(k, x, r)} < \tau(z, R), T_{Q'} < \tau(z, R)} \nonumber \\
		& = \bbP^{(T_{Q'}, S^{\phi}_{T_{Q'}})} \OBL{T_{U(k, x, r)} < \tau(z, R)} \bbP^{(0, x)} \OBL{T_{Q'} < \tau(z, R)}.
	\end{align}
	We are left to bound from below the first term in \eqref{al:strong_Markov_property}. Observe that if the process $(V_k, S^{\phi}_k)$ starts from the point $(T_{Q'}, S^{\phi}_{T_{Q'}})$ and $S^{\phi}$-coordinate stays in $B(x, r)$ for at least $\floor{\gamma / \phi(r^{-2})}$ steps, then $(V_k, S^{\phi}_k)$ hits $U(k, x, r)$ before exiting $Q(0, z, R)$. We also notice that $S^\phi$-coordinate stays in $B(x,r)$ for at least $\floor{\gamma / \phi(r^{-2})}$ steps if for all $T_{Q'} \le k \le T_{Q'} + \floor{\gamma / \phi(r^{-2})}$ it holds $\aps{S^{\phi}_k - S^{\phi}_{T_{Q'}}} < \frac{r}{2} $. 
Thus, using Theorem \ref{tm:maximal_inequality}, we get
	\begin{align*}
		\bbP^{(T_{Q'}, S^{\phi}_{T_{Q'}})}
		 \OBL{T_{U(k, x, r)} < \tau(z, R)} 
	\ge 3/4
	\end{align*}
and we conclude that
	\begin{equation*}
		\bbP^{(0, x)} \OBL{T_{U(k, x, r)} < \tau(z, R)} \ge \theta_2  \frac{j(R)}{j(r)} ,
	\end{equation*}
where $\theta_2 = \frac{3c_1}{4} $.
\end{proof}

\begin{LM}\label{lm:theta_3}
	Let $H(k, w)\geq 0$ be a function on $\calP$ such that $H(k,w)\bbjedan_{ B(x, 2r)}(w)=0 $. There exists a constant $\theta_3 > 0$ which does not depend on $x$, $r$ and $H$ and such that
	\begin{equation}\label{eq:Harnack_like_inequality}
		\bbE^{(0, x)} [H(V_{\tau(x, r)}, S^{\phi}_{\tau(x, r)})] \le \theta_3 \bbE^{(0, y)} [H(V_{\tau(x, r)}, S^{\phi}_{\tau(x, r)})],
	\end{equation}
	for all $y \in B(x, r/2)$.
\end{LM}
\begin{proof}
	It suffices to check validity of \eqref{eq:Harnack_like_inequality} for $H = \bbjedan_{(k, w)}$ if $y \in B(x, r/2)$, $w \notin B(x, 2r)$ and $1 \le k \le \floor{\gamma / \phi(r^{-2})} + 1$. With such a choice we have
	\begin{align}\label{al:hli_help1}
		\bbE^{(0, y)} [\bbjedan_{(k, w)} (V_{\tau(x, r)}, S^{\phi}_{\tau(x, r)})]
		& = \bbE^{(0, y)} [\bbE^{(0, y)} [\bbjedan_{(k, w)} (V_{\tau(x, r)}, S^{\phi}_{\tau(x, r)}) \mid \calF_{k - 1}]] \nonumber \\
		& = \bbE^{(0, y)} [\bbjedan_{\{\tau(x, r) > k - 1\}} p^{\phi}(S^{\phi}_{k - 1}, w)],
	\end{align}
	Since $S^{\phi}_{k - 1} \in B(x, r)$, we have $p^{\phi}(S^{\phi}_{k - 1}, w) \ge \inf_{z \in B(x, r)} p^{\phi}(z ,w)$. For $z \in B(x, r)$ and $w \notin B(x, 2r)$,  $z \neq w$ and whence \eqref{eq:one_step_prob_estimates} implies
	\begin{align*}
		\bbE^{(0, y)} [\bbjedan_{(k, w)} (V_{\tau(x, r)}, S^{\phi}_{\tau(x, r)})]
 \ge c_1 \bbP^{(0, y)} \OBL{\tau(x, r) = \floor{\gamma / \phi(r^{-2})} + 1} 
 \!\! \inf_{z \in B(x, r)} j(\aps{z - w}).
	\end{align*}
	If $(V_k, S^{\phi}_k)$ starts from $(0, y)$ and $S^{\phi}$-coordinate stays in $B(y, r/2)$ for $\floor{\gamma / \phi(r^{-2})}$ steps then at the same time it also stays in $B(x, r)$. Hence
	\begin{equation*}
		\frac{3}{4} \le \bbP^{(0, y)} \OBL{\max_{k \le \floor{\gamma / \phi(r^{-2})}} \aps{S^{\phi}_k - y} < \frac{r}{2}} \le \bbP^{(0, y)} (\tau(x, r) = \floor{\gamma / \phi(r^{-2})} + 1).
	\end{equation*}
	For every $z \in B(x, r)$ we have $\aps{z - w} \le 2\aps{x - w}$. By monotonicity of $j$ and \cite[Lemma 2.4]{MS18}, we get
	\begin{equation*}
		\inf_{z \in B(x, r)} j(\aps{z - w}) \ge j(2 \aps{x - w}) \ge 2^{-d - 2} j(\aps{x - w}),
	\end{equation*}
We obtain
	\begin{equation*}
		\bbE^{(0, y)} [\bbjedan_{(k, w)} (V_{\tau(x, r)}, S^\phi _{\tau(x, r)})] \ge c_2 j(\aps{x - w}).
	\end{equation*}
	Notice that \eqref{al:hli_help1} remains valid if the process starts from $(0, x)$ instead of $(0, y)$. Similarly we prove that
	\begin{align*}
		\bbE^{(0, x)} [\bbjedan_{(k, w)} (V_{\tau(x, r)}, S^{\phi}_{\tau(x, r)})]
		\le  c_3 j(\aps{x - w}).
	\end{align*}
The result follows with $\theta_3 = c_3  / c_2$.
\end{proof}

We can now prove the parabolic Harnack inequality.
\begin{proof}[Proof of Theorem \ref{tm:PHI}]
By multiplying the function $q$ by a constant, we can assume that
	\begin{equation}\label{eq:min=1}
		\min_{w \in B(z, R/B)} q(0, w) = q(0, v) = 1.
	\end{equation}
	Notice that if $q(0, x) = 0$ for some $x \in B(z, R/B)$ then \eqref{eq:PHI} is trivially satisfied, as the parabolicity of $q$ implies that
	\begin{equation*}
		\max_{(k, y) \in Q(\floor{\gamma / \phi(R^{-2})}, z, R/B)} q(k, y) = 0.
	\end{equation*}
	
Let $B$ be the constant defined at \eqref{eq:def_of_B_and_b}. By Lemma \ref{lm:R_0} of the Appendix, there exists a constant $R_0 \ge B$ such that
	\begin{equation}\label{R_0_choice}
		\floor{\gamma / \phi(r^{-2})} \ge \floor{\gamma / \phi((r/B)^{-2})} + 1, \quad  r \ge R_0.
	\end{equation}
Let us fix $r \ge R_0$, $(k, x) \in \calP$ and a set $G \subseteq Q(k + 1, x, r/B)$ for which it holds
	\begin{equation*}
		\frac{\aps{G}}{\aps{Q(k + 1, x, r/B)}} \ge \frac{1}{3}.
	\end{equation*}
We claim that for such a set $G$ there is a constant $c_1\in (0,1)$ such that
	\begin{align}\label{al:P_ge_c1}
		\bbP^{(k, x)}
		& (T_G < \tau(k, x, r)) \ge c_1.
	\end{align}
	Indeed, by our choice $G \subseteq Q(k, x, r)$ and 
$G(k) = \emptyset$. Therefore, Proposition \ref{prop:KS_estimate} and relation \eqref{eq:phi(v)/phi(u)_le_v/u} yield
	\begin{align*}
		\bbP^{(k, x)}
		(T_G < \tau(k, x, r)) 
		 \ge \frac{\theta_1}{3} \frac{\gamma}{\phi((r/B)^{-2})} c' \OBL{\frac{r}{B}}^d r^{-d} \phi(r^{-2})
		\ge \frac{\theta_1 \gamma c'}{3 B^{d + 2}} =c_1,
	\end{align*}
	where we can achieve that $c_1<1$ by decreasing $c'$ in \eqref{eq:volume_growth} if necessary.
	
Let $\theta_1, \theta_2$ and $\theta_3$ be the constants from Proposition \ref{prop:KS_estimate}, Lemma \ref{lm:theta_2} and Lemma \ref{lm:theta_3} respectively. We set
	\begin{equation}\label{eq:def_of_eta_zeta}
		\eta = \frac{c_1}{3}, \quad \zeta = \frac{c_1}{3} \wedge \frac{\eta}{\theta_3}, \quad a = 2\vee \OBL{\frac{2}{c_*}}^{1/\alpha_*},
	\end{equation}
	where $c_1$ is the constant from relation \eqref{al:P_ge_c1} and $c_*, \alpha_* \in(0, 1)$ are the constants from the scaling condition \eqref{eq:scaling}. 
\begin{Claim}\label{Claim_1}
There exists a constant $c_2>0$ such that for all $r, R, K > 0$ which satisfy
	\begin{equation}\label{eq:r_R_K_C2}
		\frac{r}{R} < 1 \quad \textnormal{and} \quad \frac{r}{R} \,K^{1/(d + 2)} \ge c_2,
	\end{equation}
	the following two inequalities hold
	\begin{align}
			\frac{j(2 \sqrt{a} R)}{j(r/R_0)} & > \frac{1}{\theta_2 \zeta K}, \label{eq:key_ineq_1} \\
\aps{Q(0, x, r/B)} j(\sqrt{b} R) &> \frac{3}{\theta_1 \zeta K}. \label{eq:key_ineq_2}
\end{align}	
\end{Claim} 
\noindent We prove this claim in the end of the proof of the theorem and the value of the constant $c_2$ is specified there, see \eqref{c_2Constant}.
\vspace*{0,2cm}
	
We construct a sequence of points $(k_i, x_i)$ such that $K_1 = q(k_1, x_1)$ is large enough and under this condition the sequence $K_i = q(k_i, x_i)$ is increasing and tends to infinity, cf. \eqref{eq:K_i+1-K_i}. This will finally contradict the fact that $q$ is bounded and whence the result will follow.
Let us choose $(k_1, x_1) \in Q(\floor{\gamma / \phi(R^{-2})}, z, R)$ such that it holds
	\begin{equation*}
		K_1 = q(k_1, x_1) = \max_{(k, y) \in Q(\floor{\gamma / \phi(R^{-2})}, z, R/B)} q(k, y).
	\end{equation*}
Evidently it suffices to study the case $c_2 K_1^{-1/(d + 2)} < 1/B$. 
Suppose that the points $(k_1, x_1), (k_2, x_2), \ldots, (k_i, x_i) \in Q(\floor{\gamma / \phi(R^{-2})}, z, R)$ are already defined.
We describe the procedure how to obtain $(k_{i + 1}, x_{i + 1}) \in Q(\floor{\gamma / \phi(R^{-2})}, z, R)$.
We first define $r_i$ by 
	\begin{equation}\label{eq:def_of_ri}
		\frac{r_i}{R} = c_2 K_i^{-1/ (d + 2)}.
	\end{equation}
With our choice of constants and using \eqref{R_0_choice} one can easily verify that for $v$ defined in \eqref{eq:min=1} it holds
\begin{align}\label{qqq}
(k_i, x_i) \in Q(0, v, \sqrt{a} R) \quad \mathrm{and } \quad k_i \ge 1 + \floor{\gamma / \phi((r_i/R_0)^{-2})}.
\end{align}

Now, suppose that $q \ge \zeta K_i$ on the set $U_i := \{k_i\} \times B(x_i, r_i / R_0)$. 
	Since $q$ is parabolic on $D =  \{0, 1, 2, \ldots, \floor{\gamma / \phi((\sqrt{b} R)^{-2})}\} \times \bbZ^d$, we know that $(q(V_{k \wedge \tau_D}, S^{\phi}_{k \wedge \tau_D}))_{k \ge 0}$ is a martingale. Thus \eqref{eq:key_ineq_1} and Lemma \ref{lm:theta_2} imply
	\begin{align*}
		1
		 = q(0, v) &= \bbE^{(0, v)} [q(V_{T_{U_i} \wedge \tau(v, 2\sqrt{a}R)}, S^{\phi}_{T_{U_i} \wedge \tau(v, 2\sqrt{a}R)})] \\
		& \ge \bbE^{(0, v)} [q(V_{T_{U_i} \wedge \tau(v, 2\sqrt{a}R)}, S^{\phi}_{T_{U_i} \wedge \tau(v, 2\sqrt{a}R)}) \bbjedan_{\{T_{U_i} < \tau(v, 2\sqrt{a}R)\}}] \\
		& = \bbE^{(0, v)} [q(V_{T_{U_i}}, S^{\phi}_{T_{U_i}}) \bbjedan_{\{T_{U_i} < \tau(v, 2\sqrt{a}R)\}}] \ge \zeta K_i \bbP^{(0, v)} (T_{U_i} < \tau(v, 2\sqrt{a}R)) \\
		& \ge \zeta K_i \theta_2 \frac{j(2\sqrt{a}R)}{j(r_i/R_0)} > \zeta K_i \theta_2 \frac{1}{\zeta K_i \theta_2 } = 1,
	\end{align*}
and we mention that we could apply Lemma \ref{lm:theta_2} because of \eqref{qqq}.
	Thus we get a contradiction, so there must exist $y_i \in B(x_i, r_i/R_0)$ such that $q(k_i, y_i) < \zeta K_i$.Observe that 
\begin{align*}
q(k_i, y_i) < \zeta K_i \le (c_1/3) K_i < K_i/3
\end{align*}	
and whence $x_i \neq y_i$. This in turn implies
	\begin{equation}\label{eq:ri_ge_R0}
		r_i \ge R_0.
	\end{equation}	
Suppose next that 
\begin{align*}
\bbE^{(k_i, x_i)} [q(V_{\tau(k_i, x_i, r_i)}, S^{\phi}_{\tau(k_i, x_i, r_i)}) \bbjedan_{\{S^{\phi}_{\tau(k_i, x_i, r_i)} \notin B(x_i, 2r_i)\}}] \ge \eta K_i.
\end{align*}
 By Lemma \ref{lm:theta_3} we have
	\begin{align*}
		\zeta K_i
		& > q(k_i, y_i) = \bbE^{(k_i, y_i)} [q(V_{\tau(k_i, x_i, r_i)}, S^{\phi}_{\tau(k_i, x_i, r_i)})] \\
		& \ge \bbE^{(k_i, y_i)} [q(V_{\tau(k_i, x_i, r_i)}, S^{\phi}_{\tau(k_i, x_i, r_i)}) \bbjedan_{\{S^{\phi}_{\tau(k_i, x_i, r_i)} \notin B(x_i, 2r_i)\}}] \\
		& \ge \theta_3^{-1} \bbE^{(k_i, x_i)} [q(V_{\tau(k_i, x_i, r_i)}, S^{\phi}_{\tau(k_i, x_i, r_i)}) \bbjedan_{\{S^{\phi}_{\tau(k_i, x_i, r_i)} \notin B(x_i, 2r_i)\}}] \\
		& \ge \frac{\eta}{\theta_3} K_i \ge \zeta K_i,
	\end{align*}
	which again gives a contradiction. Therefore
	\begin{equation}\label{eq:E<eta_Ki}
		\bbE^{(k_i, x_i)} [q(V_{\tau(k_i, x_i, r_i)}, S^{\phi}_{\tau(k_i, x_i, r_i)}) \bbjedan_{\{S^{\phi}_{\tau(k_i, x_i, r_i)} \notin B(x_i, 2r_i)\}}] < \eta K_i.
	\end{equation}
Define the set
	\begin{equation*}
		A_i = \{(j, y) \in Q(k_i + 1, x_i, r_i/B) : q(j, y) \ge \zeta K_i \}.
	\end{equation*}
	We want to apply Proposition \ref{prop:KS_estimate} for $A_i$ and $ Q(0, v, \sqrt{b}R)$. Clearly $A_i \subseteq Q(k_i + 1, x_i, r_i/B)$ and $A_i(0) = \emptyset$. Moreover, with the aid of \eqref{R_0_choice}, \eqref{eq:def_of_ri} and \eqref{eq:scaling} one can verify that $Q(k_i + 1, x_i, r_i/B) \subseteq Q(0, v, \sqrt{b}R)$.
Therefore
	\begin{align*}
		1
	 = q(0, v) &= \bbE^{(0, v)} [q(V_{T_{A_i} \wedge \tau(v, \sqrt{b}R)}, X_{T_{A_i} \wedge \tau(v, \sqrt{b}R)})] \\
		& \ge \bbE^{(0, v)} [q(V_{T_{A_i} \wedge \tau(v, \sqrt{b}R)}, X_{T_{A_i} \wedge \tau(v, \sqrt{b}R)}) \bbjedan_{\{T_{A_i} < \tau(v, \sqrt{b}R)\}}] \\
		& = \bbE^{(0, v)} [q(V_{T_{A_i}}, X_{T_{A_i}}) \bbjedan_{\{T_{A_i} < \tau(v, \sqrt{b}R)\}}] \ge \zeta K_i \bbP^{(0, v)}  (T_{A_i} < \tau(v, \sqrt{b}R)) \\
		& \ge \zeta K_i \theta_1 \aps{A_i} j(\sqrt{b}R)  
		 \ge \zeta K_i \theta_1 \frac{\aps{A_i}}{\aps{Q(k_i + 1, x_i, r_i/B)}} \frac{3}{\zeta K_i \theta_1},
	\end{align*}
	where we used \eqref{eq:key_ineq_2} in the last line. 
We conclude that
\begin{align*}
\frac{\aps{A_i}}{\aps{Q(k_i + 1, x_i, r_i/B)}} \le \frac{1}{3}.
\end{align*}	
Define next
	\begin{equation*}
		D_i = Q(k_i + 1, x_i, r_i/B) \setminus A_i\quad \mathrm{and}\quad
		M_i = \max_{Q(k_i + 1, x_i, 2r_i)} q.
	\end{equation*}
	By \eqref{eq:E<eta_Ki} combined with \eqref{al:P_ge_c1}, we obtain
	\begin{align*}
		K_i
		&  =\bbE^{(k_i, x_i)} [q(V_{T_{D_i}}, X_{T_{D_i}}) \bbjedan_{\{T_{D_i} < \tau(k_i, x_i, r_i)\}}] \\
		& \quad + \bbE^{(k_i, x_i)} [q(V_{\tau(k_i, x_i, r_i)}, X_{\tau(k_i, x_i, r_i)}) \bbjedan_{\{\tau(k_i, x_i, r_i) < T_{D_i}\}} \bbjedan_{\{X_{\tau(k_i, x_i, r_i)} \notin B(x_i, 2r_i)\}}] \\
		& \quad \quad + \bbE^{(k_i, x_i)} [q(V_{\tau(k_i, x_i, r_i)}, X_{\tau(k_i, x_i, r_i)}) \bbjedan_{\{\tau(k_i, x_i, r_i) < T_{D_i}\}} \bbjedan_{\{X_{\tau(k_i, x_i, r_i)} \in B(x_i, 2r_i)\}}] \\
		&\le   \zeta	K_i + \eta K_i + M_i (1 - \bbP^{(k_i, x_i)} (T_{D_i} < \tau(k_i, x_i, r_i))) \\
		&\le   \frac{c_1}{3} K_i + \frac{c_1}{3} K_i + M_i(1 - c_1) = \frac{2c_1}{3} K_i + M_i(1 - c_1).
	\end{align*}
	Hence $M_i / K_i \ge 1 + \rho$, where $\rho = c_1 / (3(1 - c_1)) > 0$. Finally, the point $(k_{i + 1}, x_{i + 1}) \in Q(k_i + 1, x_i, 2r_i)$ is chosen such that
	\begin{equation*}
		K_{i + 1} = q(k_{i + 1}, x_{i + 1}) = M_i .
	\end{equation*}
	This implies
	\begin{equation}\label{eq:K_i+1-K_i}
		K_{i + 1} \ge (1 + \rho) K_i.
	\end{equation}	
	which together with \eqref{eq:def_of_ri} gives
	\begin{equation}\label{eq:r_i+1-r_i}
		r_{i + 1} \le r_i (1 + \rho)^{-1/(d + 2)}.
	\end{equation}

	We want finally to show that if $K_1$ is chosen to be sufficiently large then the new point $(k_{i + 1}, x_{i + 1})$ will lie in $Q(\floor{\gamma / \phi(R^{-2})}, z, R)$.
Indeed, by iterating \eqref{eq:r_i+1-r_i}, we get
	\begin{equation}\label{eq:r_i+1-r_1}
		r_{i + 1} \le r_i (1 + \rho)^{-1/(d + 2)} \le r_{i - 1} (1 + \rho)^{-2/(d + 2)} \le \ldots \le r_1(1 + \rho)^{-i/(d + 2)}.
	\end{equation}
Using \eqref{eq:r_i+1-r_1} and scaling condition \eqref{eq:scaling} one easily shows that
	\begin{align}\label{eq:upper_bound_for_ki}
		k_{i + 1}
	 \le \floor{\gamma / \phi(R^{-2})} + \floor{\gamma / \phi((R/B)^{-2})} + \frac{5c_*^{-1}}{1 - \kappa^{2 \alpha_*}} \frac{1}{\phi(r_1^{-2})},
	\end{align}
with $\kappa = (1 + \rho)^{-1/(d + 2)}$. 
In a similar fashion we get
	\begin{align}\label{eq:upper_bound_for_|xi-z|}
		\aps{x_{i + 1} - z}
		& \le \frac{R}{B} + 2r_1\sum_{j = 0}^{\infty} ((1 + \rho)^{-1 / (d + 2)})^j = \frac{R}{B} + \frac{2r_1}{1 - \kappa}.
	\end{align}	
We next need the following easy technical result which we prove later.	
	\begin{Claim}\label{Claim_3}
	There is a constant $c_3 > 0$ such that the following two relation hold for all $R$ sufficiently large
	\begin{equation}\label{eq:k_i+1_is_OK}
	 \floor{\gamma / \phi((R/B)^{-2})} + \frac{5c_*^{-1}}{1 - \kappa^{2\alpha_*}} \frac{1}{\phi((c_3 R)^{-2})} 
		\le \floor{\gamma / \phi(R^{-2})} 
	\end{equation}
	and
	\begin{equation}\label{eq:x_i+1_is_OK}
		\frac{R}{B} + \frac{2 c_3 R}{1 - \kappa} < R.
	\end{equation}
	\end{Claim}
	
At last, let $c_3$ be a constant as in Claim \ref{Claim_3} and suppose that $K_1 \ge (c_2 / c_3)^{d + 2}$. This would mean that $r_1 \le c_3 R$. By \eqref{eq:upper_bound_for_ki}, \eqref{eq:upper_bound_for_|xi-z|} and Claim \ref{Claim_3}, $(k_{i + 1}, x_{i + 1}) \in Q(\floor{\gamma / \phi(R^{-2})}, z, R)$. However, by \eqref{eq:ri_ge_R0} $r_{i } \ge 3$ for all $i$. On the other hand, if we let $i$ tend to infinity in \eqref{eq:r_i+1-r_1}, we would obtain that $r_i$ approaches zero. This is a contradiction and whence $K_1 \le (c_2 / c_3)^{d + 2}$, which means that \eqref{eq:PHI} holds with $C_{PH} = (c_2 / c_3)^{d + 2}$ and for all $R$ large enough.
To finish the prove we are left to establish Claims \ref{Claim_1} and \ref{Claim_3}. 
\vspace*{0,2cm}

\noindent \textit{Proof of Claim \ref{Claim_1}}.	
We set
	\begin{equation}\label{c_2Constant}
		c_2 = 2 R_0 \sqrt{a} \OBL{\frac{1}{\theta_2 \zeta}}^{1/(d + 2)} \vee B\sqrt{b} \OBL{\frac{3}{\theta_1 \zeta	\gamma c'}}^{1/(d + 2)},
	\end{equation}
where $\gamma$ is the constant from Theorem \ref{tm:maximal_inequality}, $c'$ is the constant from \eqref{eq:volume_growth} and $b$ is defined in \eqref{eq:def_of_B_and_b}. 
We show that the claim is true with such a constant. We start by showing \eqref{eq:key_ineq_1}. 
Combining \eqref{eq:phi(v)/phi(u)_le_v/u} and \eqref{eq:r_R_K_C2} we get
	\begin{align*}
		\frac{j(2 \sqrt{a} R)}{j(r/R_0)}
		&  = (2 R_0 \sqrt{a})^{-d} \OBL{\frac{R}{r}}^{-d} \frac{\phi((2\sqrt{a}R)^{-2})}{\phi((r/R_0)^{-2})} 
		\ge  \frac{1}{(2 R_0 \sqrt{a})^{d + 2}} \OBL{\frac{r}{R}}^{d + 2} \\
		& > \frac{1}{(2 R_0 \sqrt{a})^{d + 2}} \frac{(2 R_0 \sqrt{a})^{d + 2}}{\theta_2 \zeta} K^{-1} = \frac{1}{\theta_2 \zeta K}.
	\end{align*}
Similarly, to prove \eqref{eq:key_ineq_2} we apply \eqref{eq:volume_growth} and \eqref{eq:phi(v)/phi(u)_le_v/u} and obtain
	\begin{align*}
		\aps{Q(0, x, r/B)} j(\sqrt{b} R)  
		& \geq  \frac{\gamma c' b^{-d/2}}{B^d} \OBL{\frac{r}{R}}^{d} \frac{\phi((\sqrt{b} R)^{-2})}{\phi((r/B)^{-2})} \\
		& \ge \frac{\gamma c'}{(B \sqrt{b})^{d + 2}} c_2^{d + 2} K^{-1} 
		> \frac{\gamma c'}{(B \sqrt{b})^{d + 2}} \frac{3 (B \sqrt{b})^{d + 2}}{\theta_1 \zeta \gamma c'} K^{-1} = \frac{3}{\theta_1 \zeta K}.
	\end{align*}

\noindent \textit{Proof of Claim \ref{Claim_3}}.	
Notice that \eqref{eq:k_i+1_is_OK} is equivalent to
	\begin{equation*}
		\frac{5c_*^{-1}}{1 - \kappa^{2\alpha_*}} \frac{1}{\phi((c_3 R)^{-2})} \le \floor{\gamma / \phi(R^{-2})} - \floor{\gamma / \phi((R/B)^{-2})}.
	\end{equation*}
	Using \eqref{al:helpful_result} and \eqref{eq:R0_big_enough} we get
	\begin{align*}
		\floor{\gamma / \phi(R^{-2})}
		 - \floor{\gamma / \phi((R/B)^{-2})} 
		\ge  \frac{\gamma}{2 \phi(B^2 R^{-2})}.
	\end{align*}
	Hence, it is enough to define $c_3$ for which
	\begin{equation}\label{eq:k_i+1_is_OK-pom}
\frac{\phi(B^2 R^{-2})}{\phi(c_3^{-2} R^{-2})} \le \frac{\gamma c_* (1 - \kappa^{2\alpha_*})}{10}.
	\end{equation}
This can be achieved by setting
	\begin{equation*}
		c_3 := B^{-1}  \left( 1\wedge  \OBL{\gamma c_*^2 (1 - \kappa^{2\alpha_*})/10}^{1/2\alpha_*} \wedge (B - 1) (1 - \kappa)/3\right).
	\end{equation*}
Indeed, with such a choice, for $R$ sufficiently large we apply the scaling condition and get
	\begin{equation*}
		\frac{\phi(B^2 R^{-2})}{\phi(c_3^{-2} R^{-2})} \le \frac{1}{c_*} (c_3 B)^{2\alpha_*}.
	\end{equation*}
	Clearly \eqref{eq:k_i+1_is_OK-pom} follows. With such $c_3$ the validity of \eqref{eq:x_i+1_is_OK} is obvious.
\end{proof}

\section{Lower bound}\label{sec:lower}
The aim of this section is to prove the global lower estimate. We use a probabilistic method based on the parabolic Harnack inequality.

\begin{TM}\label{thm:lower_bound}
Under our assumptions, for some constant $C> 0$
\begin{equation}\label{eq:lower_bound}
	p^{\phi}(n, x, y) \ge C \Big( \big(\phi^{-1}(n^{-1})\big)^{d/2} \wedge \frac{n}{\aps{x - y}^d} \phi(\aps{x - y}^{-2})\Big),
\end{equation}
for all $x, y \in \bbZ^d$, for all $n \in \bbN$.
\end{TM}

\begin{proof}
Let us set 
\begin{align*}
r_n = \frac{1}{\sqrt{\phi^{-1}(n^{-1})}},\quad n\geq 1.
\end{align*}
\textit{Near-diagonal bound}: We start by proving that there exists a constant $C > 0$ such that
	\begin{equation}\label{range1}
		p^{\phi}(n, x, y) \ge C \OBL{\phi^{-1}(n^{-1})}^{d/2},
	\end{equation}
for $n \in \bbN$ and $ \aps{x - y} \le d_1 r_n$, where $d_1>0$ is a constant to be specified.
We take $n\in \mathbb{N}$ and choose $R$ to satisfy $n = \gamma / \phi(R^{-2})$, where $\gamma$ is the constant from Theorem \ref{tm:maximal_inequality}. Let $q(k, w) = p^{\phi}(bn - k, x, w)$, where $b$ is the constant from \eqref{eq:def_of_B_and_b}.  By Lemma \ref{lm:hk_is_parabolic_function}, $q$ is parabolic on $\{0, 1, 2, \ldots, bn\} \times \bbZ^d$. Since by our choice 
$bn \ge  \gamma / \phi((\sqrt{b}R)^{-2})$, 
%
$q$ is also parabolic on $\{0, 1, 2, \ldots, \floor{\gamma / \phi((\sqrt{b}R)^{-2})}\} \times \bbZ^d$. We now choose $d_1 = 1/B$ which implies that $B(y, d_1 r_n) \subseteq B(y, R/B)$ and whence $(n, x) \in Q(\floor{\gamma / \phi(R^{-2})}, y, R/B)$. 
%
By choosing $n$ big enough we can make $R$ large enough and this allows us to  apply Theorem \ref{tm:PHI}. Thus, there is $n_0\geq 1$ such that for all $n\geq n_0$, 
	\begin{align*}
		\min_{z \in B(y, d_1 r_n)} p^{\phi}(bn, x, z)
		& \ge \min_{z \in B(y, R/B)} p^{\phi}(bn, x, z) = \min_{z \in B(y, R/B)} q(0, z) \nonumber \\
		& \ge C_{PH}^{-1} \max_{(k, z) \in Q(\floor{\gamma / \phi(R^{-2})}, y, R/B)} q(k, z)\\
		&\ge  C_{PH}^{-1} q(n, x).
	\end{align*}
Hence, by Theorem \ref{tm:on-diag_bounds},
	\begin{align*}
		\min_{z \in B(y, d_1 r_n)} p^{\phi}(bn, x, z)
		& \ge C_{PH}^{-1} q(n, x) = C_{PH}^{-1} p^{\phi}((b - 1)n, x, x) \\
		& \ge C_{PH}^{-1} c_1 \OBL{\phi^{-1}(((b - 1)n)^{-1})}^{d/2} \\
		& \ge C_{PH}^{-1} c_1 \OBL{\phi^{-1}((bn)^{-1})}^{d/2},
	\end{align*}
for all $x\in \mathbb{Z}^d$ and $n\geq n_0$. Hence, we have proved \eqref{range1} for all integers of the form $bn$ with $n\geq n_0$.
For the remaining values of $n$ 
between $bn_0$ and $b(n_0+1)$ (and so forth)	
	we use Lemma \ref{lm:lower_bound_for_P(X_1=0)} to get
	\begin{align*}
		p^{\phi}(bn + 1, x, y)
		& = \sum_{z \in \bbZ^d} p^{\phi}(bn, x, z) p^{\phi}(z, y) \ge p^{\phi}(bn, x, y) p^{\phi}(y, y) \ge C_1 p^{\phi}(bn, x, y) \\
		& \ge C_1 c_2 \OBL{\phi^{-1}((bn)^{-1})}^{d/2} \ge C_1 c_2 \OBL{\phi^{-1}((bn + 1)^{-1})}^{d/2}.
	\end{align*}
For $n < bn_0$ we apply the above procedure together with \eqref{eq:one_step_prob_estimates}, and this gives \eqref{range1} for all $n$.
\vskip 0.1in
\noindent \textit{Estimate away from the diagonal}: 
Let $j(r)$ be the function defined at \eqref{j_function}.
We now show that there is $C > 0$ such that
	\begin{equation}\label{range2}
	p^{\phi}(n, x, y) \ge C n j(\aps{x - y}),
\end{equation}
for all $n\in \mathbb{N}$ and $|x-y|\geq d_2r_n$, where a constant $d_2>0$ will be specified.
We first claim that there is a constant
$c_3>0$ such that for all $x \in \bbZ^d$ and for all $k, n \in \bbN$
\begin{equation}\label{eq:p_2n_ge_pom1}
		\bbP^x\big(\max_{j \le k} \aps{S^{\phi}_j - x} \ge c_3 r_n\big) \le \frac{1}{2} \frac{k}{n}.
	\end{equation}
	By Lemma \ref{lm:Pruit_bound_general_form} and Lemma \ref{lm:h_le_phi} we get
	\begin{align*}
	 \bbP^x(\max_{j \le k} \aps{S^{\phi}_j - x} \ge c_3 r_n) \le c_4 k \phi(c_3^{-2} r_n^{-2}).
	\end{align*}
	This is true for all constants $c_3 > 0$. We define specific constant $c_3$ as
	\begin{equation*}
		c_3 = 1\vee (2c_4 / c_*)^{1/2\alpha_*}.
	\end{equation*}
	Since $c_3 \ge 1$ we can use lower scaling to obtain \eqref{eq:p_2n_ge_pom1}.


We now set $d_2 = 3c_3$ and we notice that  $d_1 < d_2$, as $d_1 = 1/B\le 1/3$. 
Let
\begin{align*}
\tau(x, r) = \inf\{k: S^{\phi}_k \notin B(x, r)\}
\end{align*}
and consider a family of sets
	\begin{equation}\label{eq:def_of_Ak}
		A_k = \{\tau(x, c_3 r_n) = k, S^{\phi}_k, S^{\phi}_{k + 1}, \ldots, S^{\phi}_{n - 1} \in B(y, c_3 r_n), S^{\phi}_n = y\},
	\end{equation}
	for $k = 1, 2, \ldots, n$. Observe that
	\begin{equation*}
		p^{\phi}(n, x, y) = \bbP^x(S^{\phi}_n = y) \ge \sum_{k = 1}^{n} \bbP^x(A_k)
	\end{equation*}
and our task is to estimate the last sum from below. By the time reversal of the random walk we get
\begin{align}\label{al:calculation_of_P(Ak)-begin}
		\bbP^x(A_k)
		&=\sum_{\substack{x_{k - 1} \in B(x, c_3 r_n) \\ x_k \in B(y, c_3 r_n)}}  \left( \bbP^x(\tau(x, c_3 r_n) > k - 1, S^{\phi}_{k - 1} = x_{k - 1}) p^{\phi}(x_{k - 1}, x_k)\right. \nonumber \\
	 & \left. \qquad \qquad \qquad  \times \bbP^y(\tau(y, c_3 r_n) > n - k, S^{\phi}_{n - k} = x_k)\right).
\end{align}
For $x_{k - 1} \in B(x, c_3 r_n)$, $x_k \in B(y, c_3 r_n)$ and $\aps{x - y} \ge d_2 r_n = 3c_3 r_n$, we have
	\begin{align*}
		\aps{x_{k - 1} - x_k}
\le 3 c_3 r_n + \aps{x - y} \le 2\aps{x - y},
	\end{align*}
	and whence, for $\aps{x - y} \ge d_2 r_n$, by using \eqref{eq:one_step_prob_estimates}
	\begin{equation}\label{eq:p(x_k-1,x_k)}
		p^{\phi}(x_{k - 1}, x_k) \ge c_5 j(\aps{x - y}).
	\end{equation}
Thus 
	\begin{align}\label{al:calculation_of_P(Ak)-middle}
		\bbP^x(A_k) 
		\ge 
		c_5 j(\aps{x - y}) & \bbP^x(\tau(x, c_3 r_n) > k - 1) \bbP^y(\tau(y, c_3 r_n) > n - k) .
	\end{align}
Using \eqref{eq:p_2n_ge_pom1} we get
	\begin{equation*}
		\bbP^x(A_k) \ge c_5 \OBL{1 - \frac{1}{2} \frac{k - 1}{n}} \OBL{1 - \frac{1}{2} \frac{n - k}{n}} j(\aps{x - y})
		\ge 
		\frac{c_5}{4} j(\aps{x - y})
	\end{equation*}
and \eqref{range2} follows
	for all $n \in \bbN$ and $\aps{x - y} \ge d_2 r_n$.
\vskip 0.1in
\noindent \textit{Intermediate estimate}: We finally show that 
	\begin{equation}\label{eq:lower_bound_in_3rd_region}
		p^{\phi}(n, x, y) \ge C \OBL{\phi^{-1}(n^{-1})}^{d/2},
	\end{equation}
for all $n \in \mathbb{N}$ and for $d_1r_n < \aps{x - y} < d_2r_n$.
For any $1\le K\le n$ we can write
	\begin{align*}
		p^{\phi}(n, x, y)
		& \ge \sum_{z\in B(y, d_1 r_n/2)} p^{\phi}(\floor{n/K}, x, z) p^{\phi}(n - \floor{n/K}, z, y).
	\end{align*}	
We now state the claim which we prove later.
\begin{Claim}\label{Claim_4}
Let us set
	\begin{equation}\label{eq:def_of_K}
		K = 2\vee c^* \OBL{\frac{2d_2}{d_1}}^{2\alpha^*}\vee \OBL{1 - \frac{4^{-\alpha_*}}{c_*}}^{-1}.
	\end{equation}
Then for all $n\geq K$ the following inequalities hold
	\begin{align*}
		\frac{d_1 r_n }{2}  \ge d_2 r_{\floor{n/K}},\quad 
		r_{n - \floor{n/K}} \ge \frac{r_n}{2} .
	\end{align*}
\end{Claim}
\noindent Thus, if $\aps{x - y} > d_1 r_n$ and $z \in B(y, d_1 r_n/2)$ then
	\begin{align*}
		\aps{x - z}  \ge d_2 r_{\floor{n/K}} \quad \mathrm{and}\quad 
		\aps{y - z}  \le d_1 r_{n - \floor{n/K}}.
	\end{align*}
Combining this with \eqref{range1} and \eqref{range2} we get
	\begin{equation*}
		p^{\phi}(n, x, y) \ge c_6 \sum_{z \in B(y, d_1 r_n / 2)}  \floor{n/K} j(\aps{x - z}) \OBL{\phi^{-1}((n - \floor{n/K})^{-1})}^{d/2}.
	\end{equation*}	
Since $\aps{x - y} < d_2 r_n$, for every $z \in B(y, d_1 r_n / 2)$ we get $\aps{x - z} \le c_7 r_n$, where $c_7 = d_1/2 + d_2 \ge 1$. By monotonicity of $j$ and  \eqref{eq:phi(v)/phi(u)_le_v/u} we get
\begin{align*}
		j(\aps{x - z})
\geq 	
	c_7^{-d - 2} \OBL{\phi^{-1}(n^{-1})}^{d/2} n^{-1}
\end{align*}
and whence
	\begin{align}\label{final1}
		p^{\phi}(n, x, y)
		& \ge 
	c_8 \FLOOR{n/K }n^{-1} \OBL{\phi^{-1}(n^{-1})}^{d/2} \OBL{\phi^{-1}\OBL{(n - \floor{n/K})^{-1}}}^{d/2} \APS{B(y, d_1 r_n/2)} \nonumber \\
		& \geq c_9 \FLOOR{n/K}
		n^{-1} \OBL{\frac{\phi^{-1}\OBL{(n - \floor{n/K})^{-1}}}{\phi^{-1}(n^{-1})}}^{d/2} 
		\OBL{\phi^{-1}(n^{-1})}^{d/2} .
	\end{align}
Clearly $ \FLOOR{n/K} n^{-1} \ge \frac{1}{2K}$ and, by
 \eqref{eq:scaling_for_inverse},
	\begin{align*}
		\frac{\phi^{-1}\OBL{(n - \floor{n/K})^{-1}}}{\phi^{-1}(n^{-1})}
		& \ge \OBL{\frac{1}{c^*}}^{1/\alpha^*} \OBL{\frac{n - \floor{n/K}}{n}}^{-1/\alpha^*}  \\
		& \ge  \OBL{\frac{1}{c^* - c^* / (2K)}}^{1/\alpha^*}.
	\end{align*}
Combining these two bounds with \eqref{final1} we obtain \eqref{eq:lower_bound_in_3rd_region} for all $n \ge K$ and for $d_1 r_n < \aps{x - y} < d_2 r_n$.	For $n<K$ we proceed as in the end of the proof of near-diagonal bound.
\vskip 0.1in
\noindent \textit{Proof of Claim \ref{Claim_4}}.	
	Since $r_{n/K} \ge r_{\floor{n/K}}$, it is enough to find $K$ such that
	\begin{equation*}
		\frac{d_1}{2}r_n \ge d_2 r_{n/K}\Longleftrightarrow  \frac{\phi^{-1}((n/K)^{-1})}{\phi^{-1}(n^{-1})} \ge \OBL{\frac{2d_2}{d_1}}^2.
	\end{equation*}
By \eqref{eq:scaling_for_inverse}, for $n \ge K$,
	\begin{equation*}
		\frac{\phi^{-1}((n/K)^{-1})}{\phi^{-1}(n^{-1})} \ge \OBL{\frac{1}{c^*}}^{1/\alpha^*} \OBL{\frac{(n/K)^{-1}}{n^{-1}}}^{1/\alpha^*} = \OBL{\frac{K}{c^*}}^{1/\alpha^*},
	\end{equation*}
and whence $K$ has to satisfy $K \ge c^* \OBL{\frac{2d_2}{d_1}}^{2 \alpha^*}\!\!$.
Similarly, as $r_{n - \floor{n/K}} \ge r_{n - n/K}$, it is enough to have $K$ such that
	\begin{equation*}
		r_{n - n/K} \ge \frac{1}{2}r_n 
\Longleftrightarrow
		 \frac{\phi^{-1}((n - n/K)^{-1})}{\phi^{-1}(n^{-1})} \le 4.
	\end{equation*}
	We assume that $K\ge 2$ and thus \eqref{eq:scaling_for_inverse} implies
	\begin{equation*}
		\frac{\phi^{-1}((n - n/K)^{-1})}{\phi^{-1}(n^{-1})} \le \OBL{\frac{1}{c_*}}^{1/\alpha_*} \OBL{\frac{(n - n/K)^{-1}}{n^{-1}}}^{1/\alpha_*} = c_*^{-1/\alpha_*} (1 - 1/K)^{-1/\alpha_*}.
	\end{equation*}
We conclude that $K$ has to be such that $K \ge \OBL{1 - \frac{4^{-\alpha_*}}{c_*}}^{-1}$.\\
\vspace*{0,1cm}

\noindent Finally, combining inequalities \eqref{range1}, \eqref{range2} and \eqref{eq:lower_bound_in_3rd_region} we obtain \eqref{eq:lower_bound} and the proof is finished.
\end{proof}

\section{Upper bound}\label{sec:off-diagonal_bounds}

In this section we aim to prove the global upper estimates for the transition probabilities of the random walk $S^\phi _n$. Our strategy is to study the corresponding continuous time random walk and to estimate its transition kernel and hitting time of a ball, and then to use these results to get similar identities in the discrete time. The main reason why we switch to the continuous time random walk is to prove Proposition \ref{prop:MSC-3} which is a key result to establish the off-diagonal upper estimates which are our goal. Another possible approach would be to obtain the estimate for the hitting time of a ball from Proposition \ref{prop:MSC-3} directly in the discrete setting. This, however, seems to be a hard task and we do not address this problem in the present paper.
%
%

\subsection{Estimates for the continuous time random walk}
We study the continuous time version of the random walk $S^{\phi}_n$ which is constructed in the standard way, that is we take $(U_i)_{i \in \bbN}$ to be a sequence of independent, identically distributed exponential random variables with parameter $1$ which are independent of $S^{\phi}$. Let $T_0 = 0$ and $T_k = \sum_{i = 1}^k U_i$. Then we define $Y_t = S^{\phi}_n$ if $T_n \le t < T_{n + 1}$. 
Equivalently, we can take $(N_t)_{t \ge 0}$ to be a homogeneous Poisson process with intensity $1$ independent of the random walk $S^{\phi}$ and then $Y_t = S^{\phi}_{N_t}$. 
The transition probability of the process $Y$ is denoted by $q(t,x,y)=\bbP^x(Y_t = y)$.
We want to find the upper bound for $q(t,x,y)$. 
\begin{PROP}\label{prop:UHK_for_Y}
	There is a constant $c>0$ such that
	\begin{equation}\label{eq:UHK_for_Y}
		q(t, x, y) \le c \Big(\big(\phi^{-1}(t^{-1})\big)^{d/2} \wedge \frac{t}{\aps{x - y}^d} \phi(\aps{x - y}^{-2})\Big),
	\end{equation}
for all $x, y \in \bbZ^d$ and for all $t \ge 1$.
\end{PROP}

We first handle the on-diagonal part.
\begin{LM}\label{lm:UB_pom7}
	There exists a constant $C_{2} > 0$ such that for $t > 0$ and $x, y \in \bbZ^d$
	\begin{equation}\label{eq:UB_pom7}
		q(t, x, y) \le C_{2} \OBL{\phi^{-1}(t^{-1})}^{d/2}.
	\end{equation}
\end{LM}
\begin{proof}
By independence and Theorem \ref{tm:on-diag_bounds} we get
	\begin{align*}
		q(t, x, x)
		& 
		 \le e^{-t} + c_1 e^{-t} \sum_{k = 1}^{\infty} \frac{t^k}{k!} \OBL{\phi^{-1}(k^{-1})}^{d/2} \nonumber \\
		& = e^{-t} + c_1 e^{-t} \OBL{\phi^{-1}(t^{-1})}^{d/2} 
		\Big(\sum_{k > t} + \sum_{1 \le k \le t} \Big) \frac{t^k}{k!} \frac{\OBL{\phi^{-1}(k^{-1})}^{d/2}}{\OBL{\phi^{-1}(t^{-1})}^{d/2}} \nonumber \\
		& = e^{-t} + c_1 e^{-t} \OBL{\phi^{-1}(t^{-1})}^{d/2} (\Sigma_1 + \Sigma_2).
 	\end{align*}
By monotonicity, $\Sigma_1\le e^t$.
	We next find a bound for $\Sigma_2$ and after that, we will show that $e^{-t} \le c_4 \OBL{\phi^{-1}(t^{-1})}^{d/2}$ for all $t > 0$ and for some constant $c_4 > 0$. Observe that $\Sigma_2=0$ for $t<1$.
By \eqref{eq:scaling_for_inverse} we get
	\begin{equation*}
		\Sigma_2  \le c_2 t^{d/2\alpha_*} \sum_{1 \le k \le t} \frac{t^k}{k!} \frac{1}{k^{d/2\alpha_*}} \le c_3 e^t,
	\end{equation*}
where in the last inequality we applied \cite[Cor. 3]{Zn09}.
It suffices to show that 
\begin{align*}
e^{-t} \le c_4 \OBL{\phi^{-1}(t^{-1})}^{d/2},\quad t>0.
\end{align*}
For $t \ge 1$ this follows easily from \eqref{eq:scaling_for_inverse} whereas for $t \in (0, 1)$ we observe that  $e^{-t} < 1$ and $\phi^{-1}(t^{-1}) > 1$.
Finally, by the Cauchy-Schwarz inequality we obtain
	\begin{align*}
		q(t, x, y)
		& = \sum_{z \in \bbZ^d} q(t/2, x, z) q(t/2, y, z) \\
		 &\le \Big(\sum_{z \in {\bbZ^d}} q(t/2, x, z)^2\Big)^{1/2} \Big( \sum_{z \in {\bbZ^d}}  q(t/2, y, z)^2\Big)^{1/2}
		 \le 
		 C_2 \OBL{\phi^{-1}(t^{-1})}^{d/2}
	\end{align*}
and the proof of \eqref{eq:UB_pom7} is finished.
\end{proof}

Before we prove the off-diagonal estimate in \eqref{eq:UHK_for_Y} we establish a series of auxiliary results. We follow here the elaborate approach of \cite{CKW16}. 
We use the notation 
\begin{align*}
\tau^Y(x, r) = \inf\{t \ge 0 : Y_t \notin B(x, r)\}.
\end{align*}

\begin{LM}\label{lm:exit_time_estimates_for_Y}
For all $r \ge 1$ it holds
	\begin{equation*}
		\bbE^x[\tau^Y(x, r)] \asymp \frac{1}{\phi(r^{-2})}.
	\end{equation*}
\end{LM}
\begin{proof}
Let
\begin{align*}
\tau^{S^{\phi}}(x, r) =  \inf\{k \ge 0 : S^{\phi}_k \notin B(x, r)\}.
\end{align*}
By \cite[Prop. 5.4 and Lem. 5.5]{MS18},
	\begin{equation*}
		 \bbE^x[\tau^{S^{\phi}}(x, n)] \asymp \frac{1}{\phi(n^{-2})}, \quad n \in \bbN.
	\end{equation*}
Then, by Wald's identity,
	\begin{equation*}
		\bbE^x[\tau^Y(x, n)] = \bbE^x \left( U_1+\ldots +U_{\tau^{S^{\phi}}(x, n)} \right)
		 = \bbE^x[\tau^{S^{\phi}}(x, n)].
	\end{equation*}
	Hence, for every $n \in \bbN$ we have
	\begin{equation*}
		\frac{c_1}{\phi(n^{-2})} \le \bbE^x[\tau^Y(x, n)] \le \frac{c_2}{\phi(n^{-2})}.
	\end{equation*}
Finally, by monotonicity of $\phi$ and by \eqref{eq:phi(v)/phi(u)_le_v/u} we easily conclude the desired estimate.
\end{proof}

\begin{LM}\label{lm:UB_pom1}
	There exist constants $C_3, C_4 > 0$ such that 
	\begin{equation}\label{toShow}
		\bbP^x(\tau^Y(x, r) \le t) \le 1 - \frac{C_3 \phi((2r)^{-2})}{\phi(r^{-2})} + C_4 t \phi((2r)^{-2}),
	\end{equation}
for all $x \in \bbZ^d$ and for all $r, t > 0$	
\end{LM}
\begin{proof}
We first consider the case $r\in (0, 1)$. Then $Y$ exits from the ball $B(x, r)$ as soon as it jumps to some point other than $x$. Observe that 
	\begin{equation*}
		\{\tau^Y(x, r) \le t\} = \bigcup_{n = 1}^{\infty} \{T_n \le t, S^{\phi}_1 = S^{\phi}_2 = \cdots = S^{\phi}_{n - 1} = x, S^{\phi}_n \neq x\}.
	\end{equation*}
Hence
	\begin{align*}
		\bbP^x(\tau^Y(x, r) \le t)
		 = \sum_{n = 1}^{\infty} \bbP(T_n \le t) \OBL{\bbP(S^{\phi}_1 = 0)}^{n - 1} \bbP(S^{\phi}_1 \neq 0) 
		 \le t,
	\end{align*}
	where we used Lemma \ref{lm:bound_on_gamma_dist}. Choosing $C_3' =1/2 $ we have
	\begin{equation*}
1 - \frac{C_3' \phi((2r)^{-2})}{\phi(r^{-2})} \ge \frac{1}{2}.
	\end{equation*}
If we set $C_4' = 1/\phi(1/4)$ we have $t \le C_4' t \phi((2r)^{-2})$,
	Hence, for $r < 1$ we have
	\begin{equation*}
		\bbP^x(\tau^Y(x, r) \le t) \le 1 - \frac{C_3' \phi((2r)^{-2})}{\phi(r^{-2})} + C_4' t \phi((2r)^{-2}),
	\end{equation*}
	and this is precisely \eqref{toShow} with $C_3'$ and $C_4'$.

Next, assume that $r \ge 1$. Since for any $t > 0$
	\begin{equation*}
		\tau^Y(x, r) \le t + (\tau^Y(x, r) - t)\bbjedan_{\{\tau^Y(x, r) > t\}},
	\end{equation*}
by Markov property and Lemma \ref{lm:exit_time_estimates_for_Y} we get
	\begin{align*}
		\bbE^x[\tau^Y(x, r)]
		& \le  t + \bbE^x \UGL{\bbjedan_{\{\tau^Y(x, r) > t\}} \bbE^{Y_t} [\tau^Y(x, r) - t]} \\
		& \le  t + \sup_{z \in B(x, r)} \bbE^z [\tau^Y(x, r)] \bbP^x(\tau^Y(x, r) > t) \\
		& \le t + \sup_{z \in B(x, r)} \bbE^z [\tau^Y(z, 2r)] \bbP^x(\tau^Y(x, r) > t) \\
		& \le t + \frac{c_2}{\phi((2r)^{-2})} \bbP^x(\tau^Y(x, r) > t).
	\end{align*}
Using again Lemma \ref{lm:exit_time_estimates_for_Y} we have
	\begin{align*}
 \frac{c_1}{\phi(r^{-2})} \le \bbE^x[\tau^Y(x, r)] \le t + \frac{c_2}{\phi((2r)^{-2})} \bbP^x(\tau^Y(x, r) > t) 
	\end{align*}
and whence
\begin{align*}
1 - \bbP^x(\tau^Y(x, r) \le t) \ge \frac{c_1 \phi((2r)^{-2})}{c_2 \phi(r^{-2})} - \frac{t \phi((2r)^{-2})}{c_2}.
\end{align*}
If we set $C_3 = \min\{C_3', c_1/c_2\} \le 1/2$ and $C_4= \max\{C_4', 1/c_2\}$ we obtain \eqref{toShow} and the proof is finished.
\end{proof}
We now study the truncated process which is built upon the process $Y$. 
For any $\rho > 0$ we denote by $Y^{(\rho)}$ the process obtained by removing from $Y$ the jumps of size larger than $\rho$. More precisely, the process $Y^{(\rho)}$ is associated with the following Dirichlet form
	\begin{equation*}
		\calE^{(\rho)} (u, v) = \sum_{\aps{x - y} \le \rho} (u(x) - u(y)) (v(x) - v(y))p^{\phi}(x, y),
	\end{equation*}
which is defined for functions $u,v$ from the domain of the Dirichlet form of the random walk $S^\phi$, cf. \cite[Sec. 5]{Barlow_book}. We write $q^{(\rho )} (t,x,y) $ for
the transition probability of $Y^{(\rho)}$ and $Q_t^{(\rho)}$ for its  semigroup.
We will also work with killed processes. For any non-empty $D \subseteq \bbZ^d$ we denote by $(Q_t^D)_{t\ge 0}$  the semigroup of the killed process $Y^D$. Similarly we write $(Q_t^{(\rho), D})_{t\ge 0}$ for the semigroups of $Y^{(\rho), D}$.
 Let
\begin{align*}
\tau^{(\rho)}(x, r) = \inf\{t \ge 0 : Y_t^{(\rho)} \notin B(x, r)\}.
\end{align*}

\begin{LM}\label{lm:UB_pom2}
	There exist constants $C_5 \in (0, 1)$ and $C_6 > 0$ such that for any $r, t, \rho > 0$
	\begin{equation*}
		\bbP^x(\tau^{(\rho)}(x, r) \le t) \le 1 - C_5 + C_6 t \OBL{\phi((2r)^{-2}) \vee \phi(\rho^{-2})}.
	\end{equation*}
\end{LM}
\begin{proof}
	By Lemma \ref{lm:UB_pom1} and \eqref{eq:phi(v)/phi(u)_le_v/u} we get that for all $x \in \bbZ^d$ and $r, t > 0$
	\begin{equation*}
		\bbP^x(\tau^Y(x, r) \le t) \le 1 - C_3/4 + C_4 t \phi((2r)^{-2}).
	\end{equation*}
	According to \cite[Lemma 7.8]{CKW16}, for all $t > 0$
	\begin{equation}\label{Rho_comparison}
		Q_t^{B(x, r)} \bbjedan_{B(x, r)}(x) \le Q_t^{(\rho), B(x, r)} \bbjedan_{B(x, r)}(x) + c_3 t \phi(\rho^{-2}).
	\end{equation}
\noindent \textit{Remark.} 
In \cite[Lemma 7.8]{CKW16} the authors assume more restrictive assumption on the function $\phi$ then our condition \eqref{eq:scaling}, namely they require the  global scaling. The key tool to prove \eqref{Rho_comparison} is, however,  \cite[Lemma 2.1]{CKW16} which in our case is covered by Lemma \ref{lm:Lemma21_from_CKW}. \\
\vspace*{0,1cm} 

\noindent We notice that
	\begin{align*}
		Q_t^{B(x, r)} \bbjedan_{B(x, r)}(x) & = \bbE^x\UGL{\bbjedan_{B(x, r)}(Y_t) \bbjedan_{\{\tau^Y(x, r) > t\}}} = \bbP^x(\tau^Y(x, r) > t), \\
		Q_t^{(\rho), B(x, r)} \bbjedan_{B(x, r)}(x) & = \bbE^x\UGL{\bbjedan_{B(x, r)}(Y_t^{(\rho)}) \bbjedan_{\{\tau^{(\rho)}(x, r) > t\}}} = \bbP^x(\tau^{(\rho)}(x, r) > t)
	\end{align*}
and whence
	\begin{align*}
	 \bbP^x(\tau^Y(x, r) > t) \le \bbP^x(\tau^{(\rho)}(x, r) > t) + c_1t \phi(\rho^{-2}) .
	\end{align*}
This and Lemma \ref{lm:UB_pom1} imply
\begin{align*}
\bbP^x(\tau^{(\rho)}(x, r) \le t) \le 1 - \frac{C_3}{4} + C_4 t \phi((2r)^{-2}) + c_1 t \phi(\rho^{-2})
\end{align*}
and the result follows if we choose $C_5 = C_3/4 <1$ and $C_6 = C_4 + c_1$.
\end{proof}

\begin{LM}\label{lm:UB_pom3}
	There exist constants $\varepsilon \in (0, 1)$ and $C_7 > 0$ such that for $x\in \mathbb{Z}^d$ and all $r, \lambda, \rho > 0$ with 
	$\lambda \ge C_7 \phi((r \wedge \rho)^{-2})$ it holds
	\begin{equation}\label{eq:UB_pom3}
		\bbE^x\UGL{e^{-\lambda \tau^{(\rho)}(x, r)}} \le 1 - \varepsilon .
	\end{equation}
\end{LM}
\begin{proof}
By Lemma \ref{lm:UB_pom2}, for any $t > 0$ and $x \in \bbZ^d$,
	\begin{align*}
		\bbE^x \UGL{e^{-\lambda \tau^{(\rho)}(x, r)}}
		& = \bbE^x \UGL{e^{-\lambda \tau^{(\rho)}(x, r)} \bbjedan_{\{\tau^{(\rho)}(x. r) \le t\}}} + \bbE^x \UGL{e^{-\lambda \tau^{(\rho)}(x, r)} \bbjedan_{\{\tau^{(\rho)}(x, r) > t\}}} \\
		& \le \bbP^x(\tau^{(\rho)}(x, r) \le t) + e^{-\lambda t} \\
		& \le 1 - C_5 + C_6 t \OBL{\phi((2r)^{-2}) \vee \phi(\rho^{-2})} + e^{-\lambda t}.
	\end{align*}
We now choose $\varepsilon =C_5/4 \in (0,1)$. We next take $t = c_1 / \phi((r \wedge \rho)^{-2})$, for some $c_1 > 0$, in such a way that $C_6 t \phi((2r)^{-2}) + C_6 t \phi(\rho^{-2}) \le 2 \varepsilon$. Hence, we need to choose $c_1 > 0$ such that
	\begin{equation*}
		\frac{C_6 c_1 \phi((2r)^{-2})}{\phi((r \wedge \rho)^{-2})} + \frac{C_6 c_1 \phi(\rho^{-2})}{\phi((r \wedge \rho)^{-2})}  \le 2 \varepsilon.
	\end{equation*}
	Since $\phi$ is increasing,
	\begin{align*}
\frac{\phi((2r)^{-2})}{\phi((r \wedge \rho)^{-2})} \le 1\quad \mathrm{and}\quad
		\frac{\phi(\rho^{-2})}{\phi((r \wedge \rho)^{-2})} \le 1
	\end{align*}
and thus it suffices to choose $c_1 \le \varepsilon / C_6$. At last, 
we claim that there is $C_7>0$ such that for $\lambda \ge C_7 \phi((r \wedge \rho)^{-2})$ we will have $e^{-\lambda t} \le \varepsilon$. 
Indeed, with such a choice we get that $\lambda t \geq C_7 c_1$ and thus we can choose $C_7$ so big that $e^{-\lambda t} \le C_5/4 =\varepsilon$.
We finally  obtain 
	\begin{equation*}
		\bbE^x \UGL{e^{-\lambda \tau^{(\rho)}(x, r)}} \le 1 - C_5 + C_6 t (\phi((2r)^{-2}) + \phi(\rho^{-2})) + e^{-\lambda t} \le 1 - \varepsilon,
	\end{equation*}
as desired.
\end{proof}

\begin{LM}\label{lm:UB_pom4}
	There exist constants $C_8, C_9 > 0$ such that for $x \in \bbZ^d$ and $R, \rho > 0$
	\begin{equation*}
		\bbE^x\UGL{e^{-C_7 \phi(\rho^{-2}) \tau^{(\rho)}(x, R)}} \le C_8 e^{-C_9 R/\rho },
	\end{equation*}
	where $C_7 > 0$ is the constant from Lemma \ref{lm:UB_pom3}.
\end{LM}
\begin{proof}
	We first observe that if $\rho \ge R/2$ then
	we can choose $C_8$ and $C_9$ such that $C_8 \exp(-2C_9) \ge 1$ and result follows. 
Thus we study the case $\rho \in (0, R/2)$. Let $z \in \bbZ^d$, $R > 0$ be fixed. We write for simplicity $\tau = \tau^{(\rho)}(z, R)$. For any fixed $0 < r < R/2$ we set $n = \floor{R/2r}$. Let 
\begin{align*}
u(x) = \bbE^x[e^{-\lambda \tau}]\quad \mathrm{and} \quad m_k = \norm{u}_{L^{\infty}(B(z, kr))},\ \ k \in \{1, 2, \ldots, n\}.
\end{align*}
We fix $\varepsilon$ from Lemma \ref{lm:UB_pom3} and for any $0 < \varepsilon' < \varepsilon$ we choose $x_k \in B(z, kr)$ such that
	\begin{equation*}
		(1 - \varepsilon') m_k < u(x_k) = m_k.
	\end{equation*}
	Since $x_k \in B(z, kr)$ and $n = \floor{R/2r}$ it is easy to see that for any $k \le n - 1$
	\begin{equation*}
		B(x_k, r) \subseteq B(z, (k + 1)r) \subseteq B(z, R).
	\end{equation*}
	Next we consider the following function 
	\begin{equation*}
		v_k(x) = \bbE^x[e^{-\lambda \tau_k}],\quad x\in B(x_k, r),
	\end{equation*}
	where we write $\tau_k = \tau^{(\rho)}(x_k, r)$. By the strong Markov property, for any $x \in B(x_k, r)$,
	\begin{equation*}
		u(x) = \bbE^x[e^{-\lambda \tau_k} e^{-\lambda(\tau - \tau_k)}] = \bbE^x\UGL{e^{-\lambda \tau_k} \bbE^{Y_{\tau_k}^{(\rho)}} (e^{-\lambda \tau} ) } = \bbE^x\UGL{e^{-\lambda \tau_k} u(Y_{\tau_k}^{(\rho)})}.
	\end{equation*}
	Since 
$Y_{\tau_k}^{(\rho)} \in B(x_k, r + \rho)$, we get that for every $x \in B(x_k, r)$ 
	\begin{equation*}
		u(x) \le v_k(x) \norm{u}_{L^{\infty}(B(x_k, r + \rho))}.
	\end{equation*}
	It follows that for any $0 < \rho \le r$
	\begin{equation*}
		u(x_k) \le v_k(x_k) \norm{u}_{L^{\infty}(B(x_k, r + \rho))} \le v_k(x_k) m_{k + 2}.
	\end{equation*}
Since $u(x_k) \ge (1 - \varepsilon')m_k$, we have
	\begin{equation*}
		(1 - \varepsilon')m_k \le v_k(x_k) m_{k + 2}.
	\end{equation*}
In view of Lemma \ref{lm:UB_pom3}, if $\lambda \ge C_7\phi(\rho^{-2})$ and $0 < \rho \le r$ then $v_k(x_k) \le 1 - \varepsilon$. Hence
	\begin{equation*}
		m_k \le \OBL{\frac{1 - \varepsilon}{1 - \varepsilon'}}m_{k + 2}
	\end{equation*}
and iterating yields
	\begin{equation*}
		u(z) \le m_1 \le \OBL{\frac{1 - \varepsilon}{1 - \varepsilon'}}m_3 \le \OBL{\frac{1 - \varepsilon}{1 - \varepsilon'}}^2 m_5 \le \ldots \le \OBL{\frac{1 - \varepsilon}{1 - \varepsilon'}}^{n - 1} m_{2n - 1}.
	\end{equation*}
	Since $u(x) \le 1$, we have $m_{2n - 1} \le 1$. Thus
	\begin{equation*}
		u(z) \le \OBL{\frac{1 - \varepsilon}{1 - \varepsilon'}}^{n - 1}.
	\end{equation*}
Setting $2C_9 = \log \OBL{(1 - \varepsilon')/(1 - \varepsilon)}$ we get
	\begin{align*}
 \OBL{\frac{1 - \varepsilon}{1 - \varepsilon'}}^{n - 1} \le \OBL{\frac{1 - \varepsilon}{1 - \varepsilon'}}^{R/2r - 2} 
	\end{align*}
which gives
\begin{align*}
u(z) \le  C_8 \exp\OBL{-C_9 \frac{R}{r}},
\end{align*}
with $C_8 = e^{4C_9}$. 
If we set $\lambda = C_7 \phi(\rho^{-2})$ and $\rho = r$ we conclude the result.
\end{proof}
\begin{COR}\label{cor:UB_pom5}
	For any $R, \rho, t > 0$ and all $x \in \bbZ^d$
	\begin{equation*}
		\bbP^x(\tau^{(\rho)}(x, R) \le t) \le C_8 e^{-C_9 \frac{R}{\rho} + C_7 t \phi(\rho^{-2})},
	\end{equation*}
	where $C_7 > 0$ is the constant from Lemma \ref{lm:UB_pom3} and $C_8, C_9 > 0$ from Lemma \ref{lm:UB_pom4}.
\end{COR}
\begin{proof}
By Lemma \ref{lm:UB_pom4}, 
	\begin{align*}
		\bbP^x(\tau^{(\rho)}(x, R) \le t)
		& = \bbP^x \OBL{e^{-C_7 \phi(\rho^{-2}) \tau^{(\rho)}(x, R)} \ge e^{-C_7 \phi(\rho^{-2}) t}} \\
		& \le e^{C_7 \phi(\rho^{-2}) t} \bbE^x \UGL{e^{-C_7 \phi(\rho^{-2}) \tau^{(\rho)}(x, R)}} 
		 \le C_8 e^{-C_9 \frac{R}{\rho} + C_7 t \phi(\rho^{-2})},
	\end{align*}
as desired.
\end{proof}

For any $\rho > 0$ and $x, y \in \bbZ^d$, we define
\begin{equation*}
	J_{\rho}(x, y) = p^{\phi}(x, y) \bbjedan_{\{\aps{x - y} > \rho\}}.
\end{equation*}
By Meyer's decomposition and \cite[Lemma 7.2(1)]{CKW16}, the following estimate holds
\begin{equation}\label{eq:estimate_from_Meyer}
	q(t, x, y) \le q^{(\rho)}(t, x, y) + \bbE^x\Big[\int_0^t \sum_{z \in \bbZ^d} J_{\rho}(Y_s^{(\rho)}, z) q(t - s, z, y) ds\Big], \quad  x, y \in \bbZ^d.
\end{equation}

\begin{PROP}\label{prop:UB_pom6}
	There exists $C_{10} > 0$ such that for all $t, \rho > 0$ and $x \in \bbZ^d$
	\begin{equation*}
		\bbE^x\Big[\int_0^t \sum_{z \in \bbZ^d} J_{\rho}(Y_s^{(\rho)}, z) q(t - s, z, y) ds\Big] \le C_{10} t \rho^{-d} \phi(\rho^{-2}).
	\end{equation*}
\end{PROP}
\begin{proof}
By monotonicity and  \eqref{eq:one_step_prob_estimates} 
we get $J_{\rho}(x, y) \le C_{10} \rho^{-d} \phi(\rho^{-2})$, for some $C_{10}>0$.
This and symmetry imply the result.
\end{proof}

In the next Lemma we prove the upper bound for the transition kernel of the truncated process.
\begin{LM}\label{lm:UB_pom8}
	For all $t \ge 1$ and $x, y \in \bbZ^d$
	\begin{equation}\label{eq:UB_pom8}
		q^{(\rho)}(t, x, y) \le C_{11} \OBL{\phi^{-1}(t^{-1})}^{d/2} \exp \OBL{C_{12} t \phi(\rho^{-2}) - C_{13} \frac{\aps{x - y}}{\rho}},
	\end{equation}
	where $C_{11}, C_{12}, C_{13} > 0$ are constants independent of $\rho$.
\end{LM}
\begin{proof}
A direct application of \cite[Lemma 7.2(2)]{CKW16} combined with Lemma \ref{lm:Lemma21_from_CKW} and Lemma \ref{lm:UB_pom7}, imply that for all $t > 0$ and $x, y \in \bbZ^d$ we have
	\begin{equation}\label{eq:bound_for_q^rho(t,x,y)-1st_part}
		q^{(\rho)}(t, x, y) \le q(t, x, y) e^{t c_0 \phi(\rho^{-2})} \le C_2 \OBL{\phi^{-1}(t^{-1})}^{d/2} \exp(c_0 t \phi(\rho^{-2})).
	\end{equation}
	We first observe that for $\aps{x - y} < 2\rho$ relation \eqref{eq:UB_pom8} is trivial. Indeed, since 
	\begin{equation*}
 \exp\OBL{\frac{-C_{13} \aps{x - y}}{\rho}} > \exp(-2C_{13}),
	\end{equation*}
	for any $C_{13} > 0$, we get
	\begin{align}\label{al:|x-y|_le_2rho}
		q^{(\rho)} (t, x, y)
		& \le C_{2} \OBL{\phi^{-1}(t^{-1})}^{d/2} \exp(c_0 t \phi(\rho^{-2})) \frac{\exp(-2C_{13})}{\exp(-2C_{13})} \nonumber \\
		& \le C_{11} \OBL{\phi^{-1}(t^{-1})}^{d/2} \exp \Big(C_{12} t \phi(\rho^{-2}) - C_{13} \frac{\aps{x - y}}{\rho}\Big),
	\end{align}
	for any $C_{11} \ge C_{2}/\exp(-2C_{13})$, $C_{12} \ge c_0$.
	
Assume that $\aps{x - y} \ge 2\rho$. By Corollary \ref{cor:UB_pom5}, 
	\begin{align}\label{al:bound_for_Q_pom}
		Q_t^{(\rho)} \bbjedan_{B(x, r)^c}(x)
		 \le \bbP^x(\tau^{(\rho)}(x, r) \le t) 
		\le C_8 \exp \Big(-C_9 \frac{r}{\rho} + C_7 t \phi(\rho^{-2})\Big).
	\end{align}
We set $r = \aps{x - y}/2$ and write
	\begin{align*}
		q^{(\rho)}(2t, x, y)
		& = \sum_{z \in \bbZ^d} q^{(\rho)}(t, x ,z) q^{(\rho)}(t, z, y) \\
		& \le \sum_{z \in B(x, r)^c} q^{(\rho)}(t, x ,z) q^{(\rho)}(t, z, y) + \sum_{z \in B(y, r)^c} q^{(\rho)}(t, x ,z) q^{(\rho)}(t, z, y).
	\end{align*}
By \eqref{eq:bound_for_q^rho(t,x,y)-1st_part} and \eqref{al:bound_for_Q_pom} we get
	\begin{align*}
		\sum_{z \in B(x, r)^c} q^{(\rho)}(t, x ,z) q^{(\rho)}(t, z, y) 
		&\le 
		C_2 \OBL{\phi^{-1}(t^{-1})}^{d/2} 
		e^{c_0 t \phi(\rho^{-2})} \sum_{z \in B(x, r)^c} q^{(\rho)}(t, x, z) \\
		& \le C_2 C_8  \OBL{\phi^{-1}(t^{-1})}^{d/2} e^{c_0 t \phi(\rho^{-2})} e^{-C_9 \frac{r}{\rho} + C_7 t \phi(\rho^{-2})} \\
		& = C_2 C_8 \OBL{\phi^{-1}(t^{-1})}^{d/2} 
		e^{(c_0 + C_7) t \phi(\rho^{-2}) - \frac{C_9 }{2}\frac{\aps{x - y}}{\rho}}.
	\end{align*}
We can show a similar bound for $z\in B(y, r)^c$ and thus, for every $t > 0$ and $\aps{x - y} \ge 2\rho$ we have
	\begin{equation*}
		q^{(\rho)}(2t, x, y) \le 2C_2 C_8 \OBL{\phi^{-1}(t^{-1})}^{d/2} 
		e^{(c_0 + C_7) t \phi(\rho^{-2}) - \frac{C_9 }{2}\frac{\aps{x - y}}{\rho}}.
	\end{equation*}
Replacing $t$ with $t/2$ yields \eqref{eq:UB_pom8}. 
It only remains to show that
	\begin{equation}\label{eq:key_bound_for_UB_pom8}
		\frac{\phi^{-1}((t/2)^{-1})}{\phi^{-1}(t^{-1})} \le c_1,
	\end{equation}
	for some constant $c_1 > 0$.
To prove \eqref{eq:key_bound_for_UB_pom8} we have to apply scaling condition \eqref{eq:scaling_for_inverse} and this is the reason why estimate  \eqref{eq:UB_pom8} works only for $t\geq 1$. Indeed, for $t \ge 2$, by \eqref{eq:scaling_for_inverse} we get
	\begin{equation*}
		\frac{\phi^{-1}((t/2)^{-1})}{\phi^{-1}(t^{-1})} 
		\le  \OBL{\frac{2}{c_*}}^{1/\alpha_*}.
	\end{equation*}
For $1 \le t \le 2$ we simply use monotonicity and \eqref{eq:key_bound_for_UB_pom8} follows.
\end{proof}

In the rest of this section we use the notation 
\begin{align*}
r_t = \frac{1}{\sqrt{\phi^{-1}(t^{-1})}}, \quad t\geq 1.
\end{align*}

\begin{LM}
There are $N \in \bbN$ with $N > (2\alpha_* + d)/(2\alpha_*)$ and $c_1 \ge 1$ such that for all $r > 0$, $t \ge 1$ and $x \in \bbZ^d$
	\begin{equation}\label{eq:key_ineq_in_UHK_for_Y}
		\sum_{y \in B(x, r)^c} q(t, x, y) \le c_1 r^{-\theta}\OBL{\phi^{-1}(t^{-1})}^{-\theta /2},
	\end{equation}
where $0<\theta = 2\alpha_* - (2\alpha_* + d)/N$ and $\alpha_*$ is the constant from \eqref{eq:scaling}. 
\end{LM}
\begin{proof}

We first observe that for $r \le r_t$ relation \eqref{eq:key_ineq_in_UHK_for_Y} is trivially satisfied, as in this case $r_t / r \ge 1$. 

We assume that $r > r_t$.  We set
	\begin{equation}\label{eq:def_of_N}
		N = \floor{2 + d/(2\alpha_*)}
	\end{equation}
and with this $N$ we define a sequence
	\begin{equation*}
		\rho_n = 2^{n \alpha} r^{1 - 1/N} r_t^{1/N}, \quad n \in \bbN,
	\end{equation*}
	where
	\begin{equation}\label{eq:range_of_alpha}
		\left( \frac{d}{d + 2\alpha_*} \vee \frac{1}{2}\right) <\alpha < 1.
	\end{equation}
We now show that under this choice we have
	\begin{equation}\label{eq:2^nr/rho_le_rho/r_t}
		\frac{2^n r}{\rho_n} \le \frac{\rho_n}{r_t}
	\end{equation}
and 
	\begin{equation}\label{eq:tphi(rho)_le_1}
t \phi(\rho_n^{-2}) \le 1.
	\end{equation}
Indeed, \eqref{eq:2^nr/rho_le_rho/r_t} follows from \eqref{eq:def_of_N} and from the fact that  $\alpha \geq 1/2 $, and
\begin{align*}
\frac{2^n r}{\rho_n} = 2^{n(1 - \alpha)} \OBL{\frac{r}{r_t}}^{1/N}, \  \mathrm{and}\quad 
\frac{\rho_n}{r_t} = 2^{n \alpha} \OBL{\frac{r}{r_t}}^{1 - 1/N}.
\end{align*}
Similarly, \eqref{eq:tphi(rho)_le_1} follows, since under our choice we see that $\rho_n \geq r_t$.
	
Recall that by \eqref{eq:estimate_from_Meyer} and Proposition \ref{prop:UB_pom6} we have
	\begin{equation}\label{eq:F1}
		q(t, x, y) \le q^{(\rho)}(t, x, y) + C_{10} t j(\rho),
	\end{equation}
for all $\rho, t > 0$ and $x, y \in \bbZ^d$. Next, 
by Lemma \ref{lm:UB_pom8}, for all $t \ge 1$, $x, y \in \bbZ^d$ and $n \in \bbN$, we have
	\begin{equation*}
		q^{(\rho_n)}(t, x, y) \le C_{11} \OBL{\phi^{-1}(t^{-1})}^{d/2} \exp \OBL{C_{12} t \phi(\rho_n^{-2}) - C_{13} \frac{\aps{x - y}}{\rho_n}},
	\end{equation*}
	where $C_{11}, C_{12}, C_{13} > 0$ are constants independent of $\rho_n$. Hence, for all $2^n r \le \aps{x - y} < 2^{n + 1} r$ and all $t \ge 1$ we have
	\begin{equation*}
		q^{(\rho_n)}(t, x, y) \le C_{11} \OBL{\phi^{-1}(t^{-1})}^{d/2} \exp \OBL{C_{12} t \phi(\rho_n^{-2}) - C_{13} \frac{2^n r}{\rho_n}}.
	\end{equation*}
By \eqref{eq:tphi(rho)_le_1} we get
	\begin{align}\label{al:F2}
		q^{(\rho_n)}(t, x, y) \le c_2\OBL{\phi^{-1}(t^{-1})}^{d/2} \exp \OBL{-C_{13} \frac{2^n r}{\rho_n}} .
	\end{align}
	Thus, by \eqref{eq:F1} and \eqref{al:F2} we get, for $t \ge 1$ and $x \in \bbZ^d$
	\begin{align*}
		\sum_{y \in B(x, r)^c} q(t, x, y)
		\le \sum_{n = 0}^{\infty} & \sum_{2^n r \le \aps{x - y} < 2^{n + 1} r} \big(q^{(\rho_n)}(t, x, y) + C_{10} tj(\rho_n)\big) \\
		\le c_3\sum_{n = 0}^{\infty} & (2^nr)^d \OBL{\phi^{-1}(t^{-1})}^{d/2} e^{-C_{13} \frac{2^n r}{\rho_n}} \\
		& + c_4 \sum_{n = 0}^{\infty} (2^nr)^d  t\, j(\rho_n)
           =I_1 + I_2.
	\end{align*}
	We first estimate $I_2$. Since $\rho_n^{-2} \le \phi^{-1}(t^{-1}) \le 1$, we can use \eqref{eq:scaling} to get
\begin{align*}
t \phi(\rho_n^{-2}) \le \frac{1}{c_*} \OBL{\frac{r_t}{\rho_n}}^{2\alpha_*}.
\end{align*}
This implies
	\begin{align*}
		I_2
		 \le c_4 \sum_{n = 0}^{\infty} \OBL{\frac{2^n r}{\rho_n}}^d \frac{1}{c_*} \OBL{\frac{r_t}{\rho_n}}^{2\alpha_*} 
		= \frac{c_4}{c_*} \OBL{\frac{r_t}{r}}^{2\alpha_* - (2\alpha_* + d)/N} \sum_{n = 0}^{\infty} 2^{n(d - \alpha(d + 2\alpha_*))}.
	\end{align*}
By \eqref{eq:range_of_alpha}, $	d - \alpha(d + 2\alpha_*) < 0  $ and whence 
	\begin{equation}\label{eq:bound_for_I2}
		I_2 \le c_5 \OBL{\frac{r_t}{r}}^{2\alpha_* - (2\alpha_* + d)/N}.
	\end{equation}
	We proceed to  estimate $I_1$. 
	There exists a constant $c_K > 0$ such that for $x \ge C_{13}$ $e^{-x} \le c_K x^{-K}$. Applying this, we get
	\begin{equation*}
		\exp \OBL{-C_{13} \frac{2^n r}{\rho_n}} \le c_K \OBL{\frac{C_{13} 2^n r}{\rho_n}}^{-K}, \quad K > 0.
	\end{equation*}
We set
	\begin{equation*}
		K = 1 + N (d + 2\alpha_*) \vee \frac{d}{1 - \alpha}.
	\end{equation*}
For such $K$ we have $K/N > d + 2\alpha_*$ and $(1 - \alpha) K > d$ and this yields
	\begin{align}\label{al:bound_for_I1}
		I_1 
		 \le c_3 \sum_{n = 0}^{\infty} c_K C_{13}^{-K} \OBL{\frac{2^n r}{r_t}}^d \OBL{\frac{2^{n\alpha} r^{1 - 1/N} r_t^{1/N}}{2^n r}}^K
		 \le c_6 \OBL{\frac{r_t}{r}}^{2\alpha_* - (2\alpha_* + d)/N}. 
	\end{align}
	Using the definition of $\theta$, \eqref{eq:bound_for_I2}, \eqref{al:bound_for_I1} and setting $c_1 = c_5 + c_6$ we conclude  \eqref{eq:key_ineq_in_UHK_for_Y}. 
\end{proof}

\begin{LM}\label{lm:pom_result_for_UHK_for_Y}
	Assume that condition \eqref{eq:key_ineq_in_UHK_for_Y} holds with some $\theta>0$. Then there exists a constant $c_2 > 0$ such that for any ball $B(x_0, r)$ and for any $t \ge 1$
	\begin{equation*}
		\bbP^x(\tau^Y(x_0, r) \le t) \le c_2 r^{-\theta}\OBL{\phi^{-1}(t^{-1})}^{-\theta /2}, \quad  x \in B(x_0, r/4).
	\end{equation*}
\end{LM}
\begin{proof}
	For $x \in B(x_0, r/4)$, we have $B(x, 3r/4) \subseteq B(x_0, r)$. Using \eqref{eq:key_ineq_in_UHK_for_Y} we get 
	\begin{align}\label{al:pom_UHK_key_part1}
		\bbP^x(\tau^Y(x_0, r) \le t)
		& \le \bbP^x(\tau^Y(x, 3r/4) \le t) 	\nonumber \\
		& \le \bbP^x \OBL{Y_{2t} \in B(x, r/2)^c} + \sup_{\substack{z \in B(x, 3r/4)^c \\ s \le t}} \bbP^z\OBL{Y_{2t - s} \in B(x, r/2)} \nonumber \\
		& \le  \sum_{y \in B(x, r/2)^c} q(2t, x, y) + \sup_{\substack{z \in B(x, 3r/4)^c \\ s \le t}} \sum_{y \in B(z, r/4)^c} q(2t - s, z, y) \nonumber \\
		& \le c_1 \OBL{\frac{r_{2t}}{r/2}}^{\theta} + c_1 \sup_{s \le t} \OBL{\frac{r_{2t - s}}{r/4}}^{\theta}.
	\end{align}
	Since $t \ge 1$, we can use \eqref{eq:scaling_for_inverse} to obtain
	\begin{align*}
r_{2t} \le \OBL{\frac{2}{c_*}}^{1/2\alpha_*} r_t.
	\end{align*}
Since $s \le t$, we have
	\begin{equation*}
\sup_{s \le t} r_{2t - s} \le \OBL{\frac{2}{c_*}}^{1/2\alpha_*} r_t.
	\end{equation*}
With these estimates used in \eqref{al:pom_UHK_key_part1} we get
	\begin{align*}
		\bbP^x(\tau^Y(x_0, r) \le t)
		 \le c_1 2^\theta \Big(\frac{2}{c_*}\Big)^{\theta /2\alpha_*} \Big(\frac{r_t}{r}\Big)^{\theta} + c_1 4^\theta \Big(\frac{2}{c_*}\Big)^{\theta /2\alpha_*} \Big(\frac{r_t}{r}\Big)^{\theta} 
		 = c_2 \OBL{\frac{r_t}{r}}^{\theta},
	\end{align*}
for all $x \in B(x_0, r/4)$.
\end{proof}	
	
\begin{LM}
Assume that condition \eqref{eq:key_ineq_in_UHK_for_Y} holds with
$0<\theta = 2\alpha_* - (2\alpha_* + d)/N$. Then
for all $t \ge 1$, $k \ge 1$ and $\aps{x_0 - y_0} > 4k\rho$ it holds
	\begin{equation}\label{eq:bound_with_improvement}
		q^{(\rho)}(t, x_0, y_0) \le c(k) \OBL{\phi^{-1}(t^{-1})}^{d/2} \exp \OBL{c_0 t \phi(\rho^{-2})} \Big(1 + \frac{\rho}{r_t}\Big)^{-(k - 1)\theta}.
	\end{equation}
\end{LM}	
\begin{proof}
As observed in the proof of Lemma \ref{lm:UB_pom2}, for all $t > 0$,
	\begin{equation*}
		Q_t^B \bbjedan_B(x) \le Q_t^{(\rho), B} \bbjedan_B(x) + c_1 t \phi(\rho^{-2})
	\end{equation*}
	and 
	\begin{equation*}
 \bbP^x(\tau^Y(x_0, r) \le t) = 1 - Q_t^B \bbjedan_B(x).
	\end{equation*}
This and Lemma \ref{lm:pom_result_for_UHK_for_Y} imply
	\begin{equation*}
		1 - Q_t^{(\rho), B} \bbjedan_B(x) - c_1 t \phi(\rho^{-2}) \le 1 - Q_t^B \bbjedan_B(x) \le c_2 \OBL{\frac{r}{r_t}}^{-\theta}.
	\end{equation*}
	Hence
	\begin{equation}\label{eq:bound_for_1-Q}
		1 - Q_t^{(\rho), B} \bbjedan_B(x) \le c_3 \Big[\OBL{\frac{r}{r_t}}^{-\theta} + t \phi(\rho^{-2})\Big], \quad x \in B(x_0, r/4).
	\end{equation}
We now proceed to prove \eqref{eq:bound_with_improvement}.
If $\rho < r_t$ then clearly
	\begin{equation*}
\Big(1 + \frac{\rho}{r_t}\Big)^{(k - 1)\theta} < 2^{(k - 1)\theta}.
	\end{equation*}
and, by \eqref{eq:bound_for_q^rho(t,x,y)-1st_part},
	\begin{align*}
		q^{(\rho)}(t, x_0, y_0)
		& \le C_2 2^{(k - 1)\theta} \OBL{\phi^{-1}(t^{-1})}^{d/2} \exp (c_0 t \phi(\rho^{-2})) \Big(1 + \frac{\rho}{r_t}\Big)^{-(k - 1)\theta} ,
	\end{align*}
as claimed. 

	Let us now consider the case $\rho \ge r_t$. Fix $k\ge 1$, $t \ge 1$ and $x_0, y_0 \in \bbZ^d$ such that $\aps{x_0 - y_0} > 4k\rho$. Set $r = \aps{x_0 - y_0}/2 > 2k\rho$ and
	\begin{equation}\label{eq:def_of_psi}
		\psi(r, t) = c_3 \Big[\OBL{\frac{r}{r_t}}^{-\theta} + t \phi(\rho^{-2})\Big].
	\end{equation}
Notice that $\psi(r, t)$ is non-decreasing in $t$.
	We take $R = r/k > 2\rho $ and apply \cite[Lemma 7.11]{CKW16} to get
	\begin{equation*}
		Q_t^{(\rho)} \bbjedan_{B(x_0, r)^c}(x) \le 
		\Big\{c_4 \Big[ \OBL{\frac{r/k - \rho}{r_t}}^{-\theta} + t \phi(\rho^{-2})\Big]\Big\}^{k - 1}, \quad x \in B(x_0, R).
	\end{equation*}

\noindent \textit{Remark}. In our case the assumption of \cite[Lemma 7.11]{CKW16} is valid only for $t \ge 1$. Since the lemma is proven by induction, we could repeat the argument and get the same result.

\noindent Notice that
	\begin{equation*}
 \OBL{\frac{r}{k} - \rho}^{-\theta} < \rho^{-\theta}.
	\end{equation*}
	Using this and the fact that $R > \rho$, we obtain
	\begin{equation}\label{eq:bound_for_Q-part1}
		Q_t^{(\rho)} \bbjedan_{B(x_0, r)^c}(x) \le c_1(k) \Big\{\big(\frac{\rho}{r_t}\big)^{-\theta} + t \phi(\rho^{-2})\Big\}^{k - 1}, \quad  x \in B(x_0, \rho).
	\end{equation}
We notice that
	\begin{equation*}
		t\phi(\rho^{-2}) \le \frac{1}{c_*} \OBL{\frac{\rho}{r_t}}^{-\theta}, \quad  \rho  \ge r_t.
	\end{equation*}
This follows easily by \eqref{eq:scaling}.
	 Combining this with \eqref{eq:bound_for_Q-part1} we get
	\begin{equation}\label{eq:bound_for_Q-part2}
		Q_t^{(\rho)} \bbjedan_{B(x_0, r)^c}(x) \le c_2(k) \Big(\frac{\rho}{r_t}\Big)^{-(k - 1)\theta}, \quad x \in B(x_0, \rho).
	\end{equation}
Moreover, since $\rho \ge r_t$, we have
	\begin{align*}
\Big(\frac{\rho}{r_t}\Big)^{-(k - 1)\theta} \le 2^{(k - 1)\theta} \Big(1 + \frac{\rho}{r_t}\Big)^{-(k - 1)\theta}.
	\end{align*}
	Hence, by \eqref{eq:bound_for_Q-part2},
	\begin{equation}\label{eq:bound_for_Q-part3}
		Q_t^{(\rho)} \bbjedan_{B(x_0, r)^c}(x_0) \le c_3(k) \Big(1 + \frac{\rho}{r_t}\Big)^{-(k - 1)\theta}.
	\end{equation}
Further, observe that
	\begin{equation*}
		Q_t^{(\rho)} \bbjedan_{B(x_0, r)^c}(x_0) = \bbP^{x_0}(Y_t^{(\rho)} \in B(x_0, r)^c) = \sum_{z \in B(x_0, r)^c} q^{(\rho)}(t, x_0, z)
	\end{equation*}
and, by the semigroup property,
	\begin{align*}
		q^{(\rho)}(2t, x_0, y_0)
		& = \sum_{z \in \bbZ^d} q^{(\rho)}(t, x_0, z) q^{(\rho)}(t, z, y_0) \\
		& \le \sum_{z \in B(x_0, r)^c} q^{(\rho)}(t, x_0, z) q^{(\rho)}(t, z, y_0) + \sum_{z \in B(y_0, r)^c} q^{(\rho)}(t, x_0, z) q^{(\rho)}(t, z, y_0).
	\end{align*}
Using \eqref{eq:bound_for_q^rho(t,x,y)-1st_part} and \eqref{eq:bound_for_Q-part3} we obtain
	\begin{align*}
		\sum_{z \in B(x_0, r)^c} q^{(\rho)}(t, x_0, z) q^{(\rho)}(t, z, y_0)
		& \le C_2 \OBL{\phi^{-1}(t^{-1})}^{d/2} \exp (c_0 t \phi(\rho^{-2})) Q_t^{(\rho)} \bbjedan_{B(x_0, r)^c}(x_0) \\
		& \le c_4(k)\OBL{\phi^{-1}(t^{-1})}^{d/2} \exp (c_0 t \phi(\rho^{-2})) \Big(1 + \frac{\rho}{r_t}\Big)^{-(k - 1)\theta}.
	\end{align*}
Similarly, we show that
	\begin{equation*}
		\sum_{z \in B(y_0, r)^c} q^{(\rho)}(t, x_0, z) q^{(\rho)}(t, z, y_0) \le c_4(k)\OBL{\phi^{-1}(t^{-1})}^{d/2} \exp (c_0 t \phi(\rho^{-2})) \Big(1 + \frac{\rho}{r_t}\Big)^{-(k - 1)\theta}.
	\end{equation*}
This yields
	\begin{equation*}
		q^{(\rho)}(2t, x_0, y_0) \le c_5(k)\OBL{\phi^{-1}(t^{-1})}^{d/2} \exp (c_0 t \phi(\rho^{-2})) \Big(1 + \frac{\rho}{r_t}\Big)^{-(k - 1)\theta}.
	\end{equation*}
	As in the proof of Lemma \ref{lm:UB_pom8}, we can replace $2t$ with $t$ and the proof is finished.
\end{proof}

We now finally prove the upper bound for the heat kernel of the process $Y_t$.
\begin{proof}[Proof of Proposition \ref{prop:UHK_for_Y}]	
Our aim is to prove that for all $t\ge 1$ 
	\begin{equation}\label{eq:UHK_for_x_neq_y}
		q(t, x, y) \le c_1 t \aps{x - y}^{-d} \phi(\aps{x - y}^{-2}),\quad x\neq y.
	\end{equation}
We take arbitrary $x_0, y_0 \in \bbZ^d$ such that $x_0 \neq y_0$ and we set $r := \aps{x_0 - y_0}/2$. Assume that $r < r_t$. We show that in this case the on-diagonal bound from Lemma \ref{lm:UB_pom7} is smaller than the bound in \eqref{eq:UHK_for_x_neq_y}, that is
	\begin{equation}\label{eq:pom_for_UHK_for_x_neq_y_r<r_t}
		\OBL{\phi^{-1}(t^{-1})}^{d/2} \le c_2 t r^{-d} \phi(r^{-2}).
	\end{equation}
Indeed, since $1/2 \le r < r_t$, we can use Lemma \ref{lm:scaling_for_R_ge_1} (with $L = 4$) to obtain
	\begin{align*}
		\frac{\OBL{\phi^{-1}(t^{-1})}^{d/2}}{t r^{-d} \phi(r^{-2})}
		\le \frac{4^{\alpha_*}}{c_*} \OBL{\frac{r_t}{r}}^{-2\alpha_*} \OBL{\frac{r_t}{r}}^{-d} 
		\le \frac{4^{\alpha_*}}{c_*}.
	\end{align*}
Combining \eqref{eq:pom_for_UHK_for_x_neq_y_r<r_t} with Lemma \ref{lm:UB_pom7} and using \eqref{eq:phi(v)/phi(u)_le_v/u} we get
	\begin{align}\label{al:UHK_for_r<r_t}
		q(t, x_0, y_0)
		 \le C_2 c_2 2^d t \aps{x_0 - y_0}^{-d} \phi(4\aps{x_0 - y_0}^{-2}) 
		\le c_3 t \aps{x_0 - y_0}^{-d} \phi(\aps{x_0 - y_0}^{-2}).
	\end{align}

We next consider the case $r \ge r_t$. We set $k = 1 + (d + 2\alpha^*)/\theta$ and $\rho = r/(8k)$. 
By \eqref{eq:estimate_from_Meyer}, Proposition \ref{prop:UB_pom6} and \eqref{eq:bound_with_improvement},
	\begin{align*}
q(t, x_0, y_0) \le c(k) \OBL{\phi^{-1}(t^{-1})}^{d/2} \exp \OBL{c_0 t \phi(\rho^{-2})} \Big(1 + \frac{\rho}{r_t}\Big)^{-(k - 1)\theta} \!\!\!\!
+ C_{10}t \rho^{-d} \phi(\rho^{-2}).
\end{align*}
We observe that $t \phi(\rho^{-2})$ is bounded. This follows as $r\ge r_t$ implies $t\phi(r^{-2}) \le 1$, and we use $\rho = r/(8k)$ with \eqref{eq:phi(v)/phi(u)_le_v/u} to get
	\begin{equation*}
		t\phi(\rho^{-2}) =  t\phi(64k^2 r^{-2}) \le 64k^2 t\phi(r^{-2}) \le 64k^2.
	\end{equation*}
Hence
	\begin{align}\label{al:UHK_for_r_ge_r_t-part1}
		q(t, x_0, y_0)
		& \le c(k) \OBL{\phi^{-1}(t^{-1})}^{d/2} \exp(c_0 64k^2) \Big(1 + \frac{\rho}{r_t}\Big)^{-(k - 1)\theta} + C_{10} t \rho^{-d} \phi(\rho^{-2}) \nonumber \\
		& \le c_6(k) \OBL{\phi^{-1}(t^{-1})}^{d/2} \Big(1 + \frac{\rho}{r_t}\Big)^{-(k - 1)\theta} + C_{10} t \rho^{-d} \phi(\rho^{-2}).
	\end{align}
Since $\rho = r/(8k)$ and $r_t/r > 0$, we get	
\begin{equation*}
 \Big(1 + \frac{\rho}{r_t}\Big)^{-(k - 1)\theta} \le c_7(k) \Big(\frac{r}{r_t}\Big)^{-(k - 1)\theta},
	\end{equation*}
 and, by \eqref{eq:phi(v)/phi(u)_le_v/u},
	\begin{equation*}
		\rho^{-d} \phi(\rho^{-2}) = (r/(8k))^{-d} \phi \big((r/(8k))^{-2}\big) \le (8k)^{d + 2} r^{-d} \phi(r^{-2}).
	\end{equation*}
These inequalities together with \eqref{al:UHK_for_r_ge_r_t-part1} yield	\begin{align}\label{al:UHK_for_r_ge_r_t-part2}
		q(t, x_0, y_0)
		& \le c_8(k) \OBL{\phi^{-1}(t^{-1})}^{d/2} \Big(\frac{r}{r_t}\Big)^{- (k - 1)\theta} \!\!\!\! +\  c_8(k) t r^{-d} \phi(r^{-2}) \nonumber \\
		& = c_8(k) t r^{-d} \phi(r^{-2}) \Big[\frac{t^{-1}}{\phi(r^{-2})} \Big(\frac{r}{r_t}\Big)^{-2\alpha^*} + 1\Big].
	\end{align}
By $r^{-2} \le r_t^{-2} \le 1$ and \eqref{eq:scaling}, we get
	\begin{equation*}
 \frac{t^{-1}}{\phi(r^{-2})} \Big(\frac{r}{r_t}\Big)^{-2\alpha^*} \le c^*.
\end{equation*}
Thus, \eqref{al:UHK_for_r_ge_r_t-part2} implies
	\begin{align}\label{al:UHK_for_r_ge_r_t}
		q(t, x_0, y_0)
		 \le c_9(k) 2^{d + 2} t \aps{x_0 - y_0}^{-d} \phi(\aps{x_0 - y_0}^{-2}).
	\end{align}
Finally, \eqref{al:UHK_for_r<r_t} and \eqref{al:UHK_for_r_ge_r_t} yield relation \eqref{eq:UHK_for_x_neq_y} for all $t\ge 1$ and $x \neq y$. Keeping in mind  Lemma \ref{lm:UB_pom7} we conclude the result.
\end{proof}

\subsection{Full upper estimate}
In this paragraph we establish the upper bound for the transition probability of the random walk $S^\phi_n$. We follow approach of \cite{BL02}, cf. also \cite{MSC15}, which is based on the application of the hitting time estimates. We start with results for the process $Y$ and then we exploit them to obtain bounds for $S^\phi_n$. Recall that
$\tau^Y(x, r) = \inf\{t \ge 0 : Y_t \notin B(x, r)\}$.

\begin{PROP}\label{prop:MSC-1}
	There exists a constant $C_{14} > 0$ such that
	\begin{equation*}
		\bbP^x(\tau^Y(x, r) \le t) \le C_{14} t \phi(r^{-2}),
	\end{equation*}
	for all $x \in \bbZ^d$, $r > 0$ and $t\ge 1$.	
\end{PROP}
\begin{proof}
By Proposition \ref{prop:UHK_for_Y} and Lemma \ref{lm:Lemma21_from_CKW}, we get
	\begin{align*}
		\bbP^x(\aps{Y_t - x} \ge r)
		 \le c_1 t \sum_{y \in B(x, r)^c} \aps{x - y}^{-d} \phi(\aps{x - y}^{-2})  \le c_2 t \phi(r^{-2}),
	\end{align*}
	for all $x \in \bbZ^d$, $r > 0$ and $t \ge 1$.
	 For simplicity we write $\tau = \tau^Y(x, r)$. Thus, by \eqref{eq:phi(v)/phi(u)_le_v/u},
	\begin{align*}
		\bbP^x(\tau \le t)
		& = \bbP^x(\tau \le t, \aps{Y_{2t} - x} \le r/2) + \bbP^x(\tau \le t, \aps{Y_{2t} - x} > r/2) \\
		& \le \bbP^x(\tau \le t, \aps{Y_{2t} - Y_{\tau}} \ge r/2) + \bbP^x(\aps{Y_{2t} - x} > r/2) \\
		& \le \bbE^x \UGL{\bbjedan_{\{\tau \le t\}} \bbP^{Y_{\tau}} (\aps{Y_{2t - \tau} - Y_0} \ge r/2)} + c_2 2t \phi((r/2)^{-2}) \\
		& \le \bbE^x \Big[\bbjedan_{\{\tau \le t\}} \sup_{y \in B(x, r)^c} \sup_{s \le t} \bbP^y(\aps{Y_{2t - s} - y} \ge r/2)\Big] + 2c_2t \phi(4r^{-2}) \\
		& \le 2c_2 t \phi(4r^{-2}) \bbE^x \UGL{\bbjedan_{\{\tau \le t\}}} + 2c_2t\phi(4r^{-2}) 
		\le C_{14} t \phi(r^{-2}),
	\end{align*}
as desired.
\end{proof}

We use the notation
\begin{align*}
	\mathcal{T}^Y(x, r)  = \inf\{t \ge 0 : Y_t \in B(x, r)\} \quad \mathrm{and}\quad
	\mathcal{T}^{S^{\phi}}(x, r)  = \inf\{k \in \bbN_0 : S^{\phi}_k \in B(x, r)\}
\end{align*}
and we recall that $r_t = \big(\phi^{-1}(t^{-1})\big)^{-1/2}$, for $t\geq 1$.

\begin{LM}\label{lm:MSC-2}
	There exists a constant $C_{15} > 0$ such that
	\begin{equation}\label{eq:MSC-2}
		\bbP^x(\mathcal{T}^Y(y, r_t) \le t) \le C_{15} t r_t^d j(\aps{x - y}),
	\end{equation}
	for all $x, y \in \bbZ^d$ and $t \ge 1$.
\end{LM}
\begin{proof}
We first show that there is $c_1 > 0$ such that
	\begin{equation}\label{eq:MSC-2-pom1}
		\bbP^z(\tau^Y(z, c_1 r_t) > t) \ge 1/2.
	\end{equation}
Indeed, we set
	\begin{equation*}
		c_1 = 1\vee \OBL{\frac{2C_{14}}{c_*}}^{1/2\alpha_*},
	\end{equation*}
where $C_{14}$ comes from Proposition \ref{prop:MSC-1}. Using Proposition \ref{prop:MSC-1} and \eqref{eq:scaling} we get
	\begin{align*}
		\bbP^z(\tau^Y(z, c_1 r_t) \le t)
		 \le C_{14} t \phi((c_1 r_t)^{-2})  
		\le \frac{C_{14}}{c_* c_1^{2\alpha_*}} \le \frac{1}{2}.
	\end{align*}

We now consider the case $\aps{x - y} \le 2(1 + c_1)r_t$. By monotonicity of $j(r)$ and relation \eqref{eq:phi(v)/phi(u)_le_v/u}, we get
	\begin{align*}
		t r_t^d j(\aps{x - y})
		 &\ge t r_t^d j(2(1 + c_1)r_t) \ge (2(1 + c_1))^{-(d + 2)}\\
		 &\geq (2(1 + c_1))^{-(d + 2)} \bbP^x(\mathcal{T}^Y(y, r_t) \le t).
	\end{align*}
	Therefore
	\begin{equation}\label{eq:MSC-2-key_small_dist}
		\bbP^x(\mathcal{T}^Y(y, r_t) \le t) \le C_{15}' t r_t^d j(\aps{x - y}),
	\end{equation}
with $C_{15}' = (2(1 + c_1))^{d + 2}$. 

Next, we consider the case $\aps{x - y} > 2(1 + c_1)r_t$. We write $\mathcal{T} = \mathcal{T}^Y(y, r_t)$. Using the strong Markov property and \eqref{eq:MSC-2-pom1} we get
	\begin{align}\label{eq:MSC-2-key_big_dist-part1}
		\bbP^x\Big(\mathcal{T} \le t, \sup_{\mathcal{T} \le s \le \mathcal{T} + t} \aps{Y_s - Y_\mathcal{T}} \le c_1r_t\Big)
		&= \bbP^{Y_\mathcal{T}} \Big(\sup_{s \le t} \aps{Y_s - Y_0} \le c_1r_t\Big) \bbP^x(\mathcal{T} \le t)\nonumber \\
		&\ge \frac{1}{2} \bbP^x(\mathcal{T} \le t).
	\end{align}
If $\mathcal{T}  \le t$ and $\sup_{\mathcal{T}  \le s \le \mathcal{T}  + t} \aps{Y_s - Y_\mathcal{T} } \le c_1 r_t$ then $\aps{Y_t - Y_\mathcal{T} } \le c_1 r_t$.
As $\mathcal{T}$ is the first moment when the process $Y_t$ hits the ball $B(y, r_t)$, it follows that
	\begin{equation*}
		\aps{Y_t - y} \le \aps{Y_t - Y_\mathcal{T}} + \aps{Y_\mathcal{T} - y} \le c_1r_t + r_t = (1 + c_1)r_t.
	\end{equation*}
Combining these two inequalities with \eqref{eq:MSC-2-key_big_dist-part1}, we get
	\begin{align}\label{al:MSC-2-key_big_dist-part2}
		\bbP^x(\mathcal{T} \le t)
		& \le 2 \bbP^x(\aps{Y_t - y} \le (1 + c_1)r_t) 
		\le 
		2 \!\!\! \!\!\! \sum_{z \in B(y, (1 + c_1)r_t)} q(t, x, z).
	\end{align}
	Since $x \notin B(y, 2(1 + c_1)r_t)$ and $z \in B(y, (1 + c_1)r_t)$, we have $x \neq z$ and thus we can use \eqref{eq:UHK_for_x_neq_y}. Notice also that $\aps{x - z} \ge \aps{x - y}/2$.
This, monotonicity of $j$, \cite[Lemma 2.4]{MS18} and \eqref{al:MSC-2-key_big_dist-part2} imply
	\begin{align}\label{al:MSC-2-key_big_dist}
		\bbP^x(\mathcal{T}  \le t)
		\le c_2 \, t \sum_{z \in B(y, (1 + c_1)r_t)} j(\aps{x - z})
		\le C_{15}'' t r_t^d j(\aps{x - y}).
	\end{align}
Relations \eqref{eq:MSC-2-key_small_dist} and \eqref{al:MSC-2-key_big_dist} yield the result.
\end{proof}

\begin{PROP}\label{prop:MSC-3}
	There exists a constant $C_{16} > 0$ such that
	\begin{equation*}
		\bbP^x(\mathcal{T}^{S^{\phi}}(y, r_n) \le n) \le C_{16} n r_n^d j(\aps{x - y}),
	\end{equation*}
	for all $x, y \in \bbZ^d$ and $n \in \bbN$.
\end{PROP}

\begin{proof}
As before $(T_k)_{k \in \bbN_0}$ stand for the arrival times of the Poisson process $(N_t)_{t\ge 0}$ that was used to define the process $Y$. More precisely, $N_t = k$ for all $T_k \le t < T_{k + 1}$. Using the Markov inequality, we easily get that $\bbP(T_n \le 2n) \ge \frac{1}{2}$.
By independence, Lemma \ref{lm:MSC-2} and \eqref{eq:scaling_for_inverse}, we obtain
	\begin{align*}
		\frac{1}{2} \bbP^x \big(\mathcal{T}^{S^{\phi}}(y, r_n) \le n\big)
		& \le \bbP^x \big(\mathcal{T}^{S^{\phi}}(y, r_n) \le n, T_n \le 2n\big) \le \bbP^x \big(\mathcal{T}^Y(y, r_n) \le 2n\big) \\
		& \le \bbP^x \big(\mathcal{T}^Y(y, r_{2n}) \le 2n\big) \le 2 C_{15} n r_{2n}^d j(\aps{x - y}) = C_{16} n r_n^d j(\aps{x - y}),
	\end{align*}
as claimed.
\end{proof}

In the following theorem we finally prove the upper bound for the transition probability of the random walk $S^{\phi}$. In the proof we again apply the parabolic Harnack inequality.

\begin{TM}\label{tm:upper_bound_for_Sphi}
	There exists a constant $C > 0$ such that
	\begin{equation*}
		p^{\phi}(n, x, y) \le C \Big(\big(\phi^{-1}(n^{-1})\big)^{d/2} \wedge \frac{n}{\aps{x - y}^d} \phi(\aps{x - y}^{-2})\Big),
	\end{equation*}
for all $x, y \in \bbZ^d$ and $n \in \bbN$.
\end{TM}

\begin{proof}
	By Proposition \ref{prop:MSC-3} we have for all $k \in \bbN$
	\begin{equation*}
		\sum_{z \in B(y, r_k)} p^{\phi}(k, x, z) \le \bbP^x(\mathcal{T}^{S^{\phi}}(y, r_k) \le k) \le C_{16} k r_k^d j(\aps{x - y}).
 	\end{equation*}
 	On the other hand
 	\begin{equation*}
 		\sum_{z \in B(y, r_k)} p^{\phi}(k, x, z) \ge c' r_k^d \min_{z \in B(y, r_k)} p^{\phi}(k, x, z).
 	\end{equation*}
Hence
	\begin{equation}\label{eq:upper_bound_for_Sphi-pom1}
\min_{z \in B(y, r_k)} p^{\phi}(k, x, z) \le c_1 k j(\aps{x - y}).
 	\end{equation}
 Next we apply the parabolic Harnack inequality. 
We choose $R > 0$ to satisfy $\gamma / \phi(R^{-2}) = n$, where $\gamma$ is the constant from Theorem \ref{tm:maximal_inequality}. Remember that we can choose $\gamma$ to be even smaller than specified in the theorem. Thus we take  $\gamma \le B^{-2}$ where $B$ is the constant defined in \eqref{eq:def_of_B_and_b}. By \eqref{eq:phi(v)/phi(u)_le_v/u} we easily get that $r_n \le R/B$.
 By Lemma \ref{lm:hk_is_parabolic_function}, the function $q(k, w) = p^{\phi}(bn - k, x, w)$ is parabolic on $\{0, 1, 2, \ldots, bn\} \times \bbZ^d$, where $b$ is defined at \eqref{eq:def_of_B_and_b}.
With our choice $bn \ge \floor{\gamma / \phi((\sqrt{b}R)^{-2})}$ and thus the  function $q$ is parabolic on $\{0, 1, 2, \ldots, \floor{\gamma / \phi((\sqrt{b}R)^{-2})}\} \times \bbZ^d$. By \eqref{eq:upper_bound_for_Sphi-pom1}, we get
 	\begin{equation}\label{eq:upper_bound_for_Sphi-pom2}
		\min_{z \in B(y, R/B)} q(0, z) = \min_{z \in B(y, R/B)} p^{\phi}(bn, x, z) \le \min_{z \in B(y, r_n)} p^{\phi}(bn, x, z) \le c_1 bn j(\aps{x - y}).
 	\end{equation}
Choosing $n$ big enough we can enlarge $R$ so that we can apply Theorem \ref{tm:PHI}. Hence
\begin{equation*}
 		\max_{(k, z) \in Q(\floor{\gamma / \phi(R^{-2})}, y, R/B)} q(k, z) \le C_{PH} \min_{z \in B(y, R/B)}q(0, z).
 	\end{equation*}
 	Since $n = \gamma/\phi(R^{-2})$, it is clear that $(n, y) \in Q(\floor{\gamma / \phi(R^{-2})}, y, R/B)$. Combining this with \eqref{eq:upper_bound_for_Sphi-pom2}, we obtain 
 		\begin{align}\label{al:upper_bound_for_(b-1)n}
 		p^{\phi}((b - 1)n, x, y)
 		& = q(n, y) \le \max_{(k, z) \in Q(\floor{\gamma / \phi(R^{-2})}, y, R/B)} q(k, z) \le C_{PH} \min_{z \in B(y, R/B)}q(0, z) \nonumber \\
 		& \le C_{PH} c_1 bn j(\aps{x - y}) 
 		=c_2 (b - 1)n j(\aps{x - y}).
 	\end{align}
 Similarly as in the proof of Theorem \ref{thm:lower_bound}, we can show that this is enough to get the desired upper bound for all $n \in \bbN$. 
Finally, we have
	\begin{equation*}
		p^{\phi}(n, x, y) \le c_3 n j(\aps{x - y}),
	\end{equation*}
	for all $x, y \in \bbZ^d$, $x \neq y$ and $n \in \bbN$. This combined with Corollary \ref{rem:off-diag_bd_from_on-diag} yields the result.
\end{proof}

\section{Appendix}

\begin{LM}\label{lm:scaling_for_R_ge_1}
	Let $L \ge 1$. Then for all $0 < r \le 1\wedge R \le R \le L$ we have	\begin{equation}\label{eq:scaling_for_R_ge_1}
		\frac{c_*}{L^{\alpha_*}} \OBL{\frac{R}{r}}^{\alpha_*} \le \frac{\phi(R)}{\phi(r)} \le \phi(L) c^* \OBL{\frac{R}{r}}^{\alpha^*}.
	\end{equation}
\end{LM}
\begin{proof}
	Since $L \ge 1$, relation \eqref{eq:scaling_for_R_ge_1} follows directly from \eqref{eq:scaling} in the case $R \le 1$. For $0 < r \le 1 < R \le L$ (using \eqref{eq:scaling} and the fact that $\phi$ is increasing) we have
	\begin{equation*}
		\frac{\phi(R)}{\phi(r)} \le \frac{\phi(L)}{\phi(r)}  \le \phi(L) c^* \OBL{\frac{1}{r}}^{\alpha^*} \le \phi(L) c^* \OBL{\frac{R}{r}}^{\alpha^*},
	\end{equation*}
and similarly	
	\begin{equation*}
		\frac{\phi(R)}{\phi(r)} \ge \frac{\phi(1)}{\phi(r)} \ge c_* \OBL{\frac{1}{r}}^{\alpha_*} \ge \frac{c_*}{L^{\alpha_*}} \OBL{\frac{R}{r}}^{\alpha_*},
	\end{equation*}
as desired.
\end{proof}

\begin{LM}\label{lm:R_0}
	There exists a constant $R_0 \ge B$ such that
	\begin{equation*}
		\floor{\gamma / \phi(R^{-2})} \ge \floor{\gamma / \phi((R/B)^{-2})} + 1, \quad  R \ge R_0,
	\end{equation*}
	where $B$ is defined at \eqref{eq:def_of_B_and_b}.
\end{LM}
\begin{proof}
	For every $x \in \bbR$ we write $\floor{x} = x - m(x)$, $m(x)\in [0,1)$. 
Thus, we look for $R_0$ such that
	\begin{align*}
\frac{\gamma}{\phi(R^{-2})} - \frac{\gamma}{\phi(B^2 R^{-2})} & \ge 1 + m(\gamma / \phi(R^{-2})) - m \big(\gamma / \phi((R/B)^{-2})\big), \quad  R \ge R_0.
	\end{align*}
	Observe that $1 + m(\gamma / \phi(R^{-2})) - m \big(\gamma / \phi((R/B)^{-2})\big) \le 2.$. Hence, it is enough to find $R_0$ large enough and such that
	\begin{equation*}
		\frac{\gamma}{\phi(R^{-2})} - \frac{\gamma}{\phi(B^2 R^{-2})} \ge 2, \quad  R \ge R_0.
	\end{equation*}
	By \eqref{eq:scaling}, we get
	\begin{align}\label{al:helpful_result}
		\frac{\gamma}{\phi(R^{-2})}
		 - \frac{\gamma}{\phi(B^2 R^{-2})} 
		 \ge \frac{\gamma}{\phi(B^2 R^{-2})} \OBL{c_* B^{2\alpha_*} - 1} 
		  \ge  \frac{\gamma}{\phi(B^2 R^{-2})} \overset{R \to \infty}{\longrightarrow} \infty.
	\end{align}
	Therefore, there exists $R_0 \ge B$ such that
	\begin{equation}\label{eq:R0_big_enough}
		\frac{\gamma}{\phi(B^2 R^{-2})} \ge 2, \quad  R \ge R_0
	\end{equation}
and the proof is finished.	
\end{proof}

\begin{LM}\label{lm:bound_on_gamma_dist}
	Let $(U_i)_{i \in \bbN}$ be a sequence of independent, identically distributed exponential random variables with parameter $1$ and let $T_n = \sum_{i = 1}^n U_i$. Then for all $n \in \bbN$ and $t > 0$ 
	\begin{equation*}
		\bbP(T_n \le t) \le t.
	\end{equation*}
\end{LM}
\begin{proof}
	Denote by $F_{T_n}(t) = \bbP(T_n \le t)$ the distribution function and by $f_{T_n}$ the density of $T_n$. 
It is enough to prove that $f_{T_n}(t) \le 1$, for $t > 0$.
For $n = 1$ the result is obvious. For $n \ge 2$ it is easy to check that the function $f_{T_n}$ obtains maximum for $t = n - 1$ and that 
	\begin{equation*}
		\max f_{T_n} = \frac{(n - 1)^{n - 1} e^{-(n - 1)}}{(n - 1)!}.
	\end{equation*}
The result follows from the inequality $n!\geq \sqrt{2 \pi n} n^n e^{-n}$.
\end{proof}

\noindent \textbf{Acknowledgement}.
The authors wish to express their gratitude to Z. Vondra\v{c}ek for his help and encouragement. This article was started at Graz University of Technology and we want to thank W. Woess for his hospitality.
We also thank A. Bendikov, P. Kim, R. Schilling and J. Wang for stimulating discussions.

This work has been financially supported by \textit{Deutscher Akademischer Austauschdienst} (DAAD) and \textit{Ministry of Science and Education of the Republic of Croatia} (MSE) via project \textit{Random Time-Change and Jump Processes}. W.\ Cygan was supported by \textit{National Science Centre (Poland)} under grant 2015/17/B/ST1/00062. S.\ \v{S}ebek was supported by \textit{Croatian Science Foundation} under project 3526.

\bibliographystyle{babamspl}
\bibliography{HKE}

\end{document}